\PassOptionsToPackage{dvipsnames}{xcolor}
\documentclass[11pt]{amsart}
\usepackage{verbatim}
\usepackage{graphicx}
\usepackage{bbding}
\usepackage[shortlabels]{enumitem}
\setlist[enumerate,1]{label={\upshape(\arabic*)},leftmargin=1cm}
\usepackage[margin=1in]{geometry}
\usepackage{tikz}
\usetikzlibrary{cd}
\usetikzlibrary{calc,intersections,through,backgrounds}
\usepackage{pgfplots}
\pgfplotsset{compat=1.5.1}
\usepackage{amssymb}
\usepackage[mathscr]{euscript} %for \mathscr

\numberwithin{equation}{section}

\newcommand{\sz}{\scriptsize}

\newcounter{AEcount}
\newcommand{\aen}[1] % alphabetical enumeration
 {
 \begin{list}{(\alph{AEcount})}
 {\usecounter{AEcount}}
 #1
 \end{list}
 }

\newtheorem{thm}[equation]{Theorem}

\newtheorem{lem}[equation]{Lemma}
\newtheorem*{lem*}{Lemma}
\newtheorem{cor}[equation]{Corollary}
\newtheorem{prop}[equation]{Proposition}

\newtheorem*{prop*}{Proposition}

\newtheorem*{thm**}{Theorem\theoremnum}
\newenvironment{thm*}[1][]{%
  \edef\theoremnum{\if\relax\detokenize{#1}\relax\else~#1\fi}% Store theorem number
  \begin{thm**}
}{%
  \end{thm**}
}
\theoremstyle{definition}
\newtheorem{defn}[equation]{Definition}
\newtheorem{conj}[equation]{Conjecture}
\newtheorem{exmp}[equation]{Example}

\theoremstyle{remark}

\newtheorem{rem}[equation]{Remark} %\renewcommand{\therem}{}

\newtheorem{claim}[equation]{Claim}

\newcommand{\al}{\alpha}
\newcommand{\A}{\mathscr{A}}

\newcommand{\B}{\mathscr{B}}

\newcommand{\de}{\delta}

\newcommand{\e}{\emptyset}

\newcommand{\ep}{{\varepsilon}}
\newcommand{\F}{\mathscr{F}}

\newcommand{\I}{\mathcal{I}}

\newcommand{\lb}{\lambda}

\newcommand{\m}{\medskip}

\newcommand{\n}{\noindent}

\newcommand{\R}{\mathbb{R}}

\newcommand{\sm}{\setminus}

\newcommand{\sq}{\subseteq}

\newcommand{\Z}{{\mathbb Z}}

\begin{document}

\title{Comonotone approximation and interpolation by entire functions II}

\author{Maxim R.~Burke}

\address{School of Mathematical and Computational Sciences, University of Prince Edward Island
Charlottetown PE, Canada C1A 4P3}

\email{burke@upei.ca}

\thanks{Research supported by NSERC. The author thanks Yinhe Peng and the Academy of Mathematics and Systems Science of the Chinese Academy of Sciences for their hospitality in the fall of 2024 when some of this research was carried out.}

%26A48 Monotonic functions, generalizations
%30E10 Approximation in the complex plane
%41A05 Interpolation in approximation theory
%41A28 Simultaneous approximation
%41A10 Approximation by polynomials


\begin{abstract}
A theorem of Hoischen states that given a positive continuous function $\varepsilon:\mathbb{R}\to\mathbb{R}$, an integer $n\geq 0$, and a closed discrete set $E\subseteq\mathbb{R}$, any $C^n$ function $f:\mathbb{R}\to\mathbb{R}$ can be approximated by an entire function $g$ so that for $k=0,\dots,n$, and $x\in\R$,
$|D^{k}g(x)-D^{k}f(x)|<\varepsilon(x)$, and if $x\in E$ then
$D^{k}g(x)=D^{k}f(x)$.
The approximating function $g$ is entire and hence piecewise monotone. Building on earlier work, for $n\leq 3$, we determine conditions under which when $f$ is piecewise monotone we can choose $g$ to be comonotone with $f$ (increasing and decreasing on the same intervals), and under which the derivatives of $g$ can be taken to be comonotone with the corresponding derivatives of $f$ if the latter are piecewise monotone. The proof for $n\leq 3$ establishes the theorem for all $n$, assuming a conjecture (shown in previous work with Haris and Madhavendra to hold for $n\leq 3$) regarding the set of $2(n+1)$-tuples $(f(0),Df(0),\dots,D^nf(0),f(1),Df(1),\dots,D^nf(1))$ of the values at the endpoints of the derivatives of a $C^n$ function $f$ on $[0,1]$ for which $D^nf$ is increasing and not constant.
\end{abstract}

\date{\today}

\maketitle

\section{Introduction}
\label{s:Introduction}

In this paper we continue the investigation begun in \cite{Bu2019} into comonotone approximation and interpolation of $C^n$ functions and their derivatives by entire functions when the derivatives are piecewise monotone.
Comonotone approximation by polynomials of functions on a compact interval has been treated extensively. See \cite{Gal}, Chapter 1. See also \cite{DZ} where characterizations are provided for the best comonotone polynomial approximation to a piecewise monotone continuous function on a compact interval, and the references therein. See the introduction to \cite{Bu2019} for historical information on the approximation of monotone functions by monotone entire functions on the real line.

Our work is motivated by the following result of Hoischen.

\begin{thm}\label{Hoischen}
{\rm\cite{Ho1975}} Let $n$ be a nonnegative integer.
Let $E\sq \R$ be a closed discrete set.
Suppose $f\colon\R\to\R$ is a $C^n$ function and
$\ep\colon\R\to\R$ is a positive continuous function. Then there
exists an entire function $g$ such that $g(\R)\sq \R$ and for all
$k=0,\dots,n$ and all $x\in\R$, $|D^kg(x)-D^kf(x)|<\ep(x)$ and
moreover, if $x\in E$ then $D^kf(x)=D^kg(x)$.
\end{thm}

The approximating function $g$ and its derivatives $D^jg$ are entire and hence piecewise monotone. We determine conditions under which when a $C^n$ function $f$ and its derivatives up to order $n$ are piecewise monotone we can choose $g$ so that $D^jg$ is comonotone with $D^jf$, $j=0,\dots,n$, with interpolation on the turning points.

Given a continuous functions $f\colon I\to\R$ on an interval $I$ of $\R$,
define an equivalence relation on $I$ by writing $a\sim b$ if $f$ is constant on the closed subinterval from $a$ to $b$. The equivalence classes of $I$ for this equivalence relation will be called the \emph{platforms} of the function $f$. (These are the connected components of the fibers of $f$.) The platforms are intervals and, because $f$ is continuous, they are closed. If $P=[a,b]\sq I$ is a compact platform on which $f$ has constant value $c$, and for some $\ep>0$ we have $P_\ep = (a-\ep,b+\ep)\sq I$ and either $f(x)>c$ for all $x\in P_\ep\sm P$ or $f(x)<c$ for all $x\in P_\ep\sm P$, then $P$ is called a \emph{turning platform}. The points belonging to a turning platform will be called \emph{turning points}.

We use the words increasing, decreasing and monotone in the non-strict sense, i.e., $f$ is \emph{increasing} on $I$ if $f$ is $\leq$-increasing ($x\leq y\Rightarrow f(x)\leq f(y)$), $f$ is \emph{decreasing} on $I$ if
$f$ is $\leq$-decreasing ($x\leq y\Rightarrow f(y)\leq f(x)$), and $f$ is \emph{monotone} if $f$ is increasing or decreasing. Two functions on $I$ have the \emph{same monotonicity} if they are both increasing or both decreasing.
We say that $f$ is \emph{piecewise monotone} if there is a set $K\sq I$ which is a closed discrete subset of $I$ and is such that $f$ is monotone on each component of $I\sm K$.
$K$ will be called a \emph{witnessing set} to the piecewise monotonicity of $f$.

For continuous piecewise monotone functions, the failure of monotonicity is witnessed by the existence of a turning platform.

\begin{prop}[\cite{Bu2019}, Proposition 2.1]\label{p:not piecewise monotone}
Let $I$ be a nontrivial interval of $\R$ and let $f\colon I\to\R$ be continuous and piecewise monotone. Then $f$ has a turning platform if and only if $f$ is not monotone.
\end{prop}

Real-analytic functions are piecewise monotone on the real line. Proposition \ref{p:summarize} records the (standard) fact that more generally $C^1$ functions whose derivative has a closed discrete set of zeros are piecewise monotone and the fact that the nonexistence of flat points for a $C^n$ function is a sufficient condition to ensure that the zero set is closed discrete.

\begin{defn}[\cite{Bu2019}, Definition 2.4]\label{def:flat points}
With $n$ a nonnegative integer or $\infty$, we say that $p\in I$ is an \emph{$n$-flat point}, or just a \emph{flat point}, for a $C^n$ function $f\colon I\to\R$ if $D^kf(p)=0$, $0\leq k\leq n$, $k\in \Z$.
\end{defn}

Notice that a flat point for $f$ is also a flat point for each $D^kf$, $0\leq k\leq n$, $k\in\Z$.\footnote{When $n$ is finite, this sentence says  more precisely ``Notice that an $n$-flat point for $f$ is an $(n-k)$-flat point for $D^kf$, $0\leq k\leq n$, $k\in\Z$.''} Thus, if for some nonnegative integer $k\leq n$, $D^kf$ has no flat points, then none of the derivatives $D^if$, $i=0,\dots,k$, have flat points either.

\begin{prop}[\cite{Bu2019}, Propositions 2.5, 2.6, 2.7]\label{p:summarize}
Let $n$ be a nonnegative integer or $\infty$, $k$ a nonnegative integer with $k\leq n$. Let $I$ be a nontrivial interval of $\R$ and let $f\colon I\to\mathbb{R}$. Write $Z_i$ for the zero set of $D^if$ (when $D^if$ exists), and write $E_i$ for the set of turning points for $D^if$.
\begin{enumerate}
\item
If $f$ is a $C^n$ function and has no flat points, then $Z_0$ is closed discrete in $I$.

\item
If $f$ is a $C^1$ function and $Z_1$ is closed discrete, then $E_0$ is closed discrete in $I$ and $f$ is strictly monotone on the components of $I\sm E_0$.

\item
If $f$ is a $C^n$ function and $D^kf$ has no flat points, then $\bigcup_{i=0}^{k}Z_i$ is closed discrete in $I$ and, when $0\leq i<k$, $E_i$ is closed discrete in $I$ and $D^if$ is strictly monotone on the components of $I\sm E_i$.
\end{enumerate}
\end{prop}

Two functions $f,g\colon I\to\R$ are \emph{comonotone with witnessing set $K$} if they are both piecewise monotone with $K$ as a witnessing set, and $f$ and $g$ have the same monotonicity on each component of $I\sm K$.

\begin{rem}\label{r:comonotone}
(a) Because constant functions are monotone, comonotone functions need not be monotone on the same intervals. For example, the functions $f(x)=0$ and $g(x)=x^2$ on $\R$ are comonotone with witnessing set $\{0\}$, but $f$ is monotone on $\R$ whereas $g$ is not.

(b) Even when comonotone functions are monotone on the same intervals, they can be comonotone for some witnessing sets and not for others. For example, the functions $f,g\colon\R\to\R$ given by $f(x)=\max(x,0)$, $g(x)=\max(-x,0)$ are both monotone and hence every closed discrete set $K$ is a common witness to their piecewise monotonicity. However, $f$ and $g$ are comonotone using a closed discrete set $K$ as a witnessing set if and only if $0\in K$.
\end{rem}

The following proposition shows that when the turning platforms of $f$ are singletons, and $f$ is strictly monotone on the complementary intervals, the witnessing set for the comonotonicity of $f$ and another piecewise monotone function $g$ can be taken to be the set of turning points for $f$. Hence, in this case the mention of the witnessing set can be omitted, and we can speak unambiguously of a function being comonotone with $f$.

\begin{prop}[\cite{Bu2019}, Proposition 2.8]\label{p:witness.doesn't.matter}
Let $I$ be a nontrivial interval of $\R$ and let $f,g\colon I\to\R$ be continuous piecewise monotone functions. Suppose that the turning platforms of $f$ are singletons and set
\[
K_0=\{x\in I:x\ \text{is a turning point for}\ f\}.
\]
If $f$ is strictly monotone on the components of $I\sm K_0$, then for any common witnessing set $K$ to their piecewise monotonicity, $f$ and $g$ are comonotone with witnessing set $K$ if and only if they are comonotone with witnessing set $K_0$.
\end{prop}

Proposition \ref{p:summarize} gave a simple assumption on a $C^1$ function under which the hypothesis of Proposition \ref{p:witness.doesn't.matter} is satisfied, namely the discreteness of the set of zeros for $Df$. Under this same assumption, we get a simple criterion for comonotonicity of a function $g$ with $f$.

\begin{prop}[\cite{Bu2019}, Proposition 2.9]\label{p:no.flat.points.comonotone}
Let $I$ be a nontrivial interval of $\R$ and let $f,g\colon I\to\mathbb{R}$ be $C^1$ functions.
If the zero set of $Df$ is a closed discrete set in $I$, and $Df$ and $Dg$ have the same sign everywhere on $I$, then $g$ is comonotone with $f$.
\end{prop}

\begin{prop}\label{witnessing set}
Let $K$ be a witnessing set for a piecewise monotone continuous function $f$ on a nontrivial interval $I$.
\begin{enumerate}[\rm(a)]
\item
$K$ is $\sq$-minimal if and only if the points of $K$ are all turning points and $K$ has exactly one point on each turning platform.

\item
$K$ contains a $\sq$-minimal witnessing set.
\end{enumerate}
\end{prop}

\begin{proof}
$K$ must have at least one point on each turning platform since $f$ is not monotone in any neighborhood of a turning platform. Choose one point $x_P\in K\cap P$ for each turning platform $P$.

\m

\n \emph{Claim}. $K'=\{x_P:P$ is a turning platform$\}$ is also a witnessing set.

\m

To prove the claim, note that if $K'$ is not a witnessing set, then there is a component $J$ of $I\sm K'$ such that $f$ is not monotone on $J$. Then the restriction of $f$ to $J$ has a turning platform $P$ by Proposition \ref{p:not piecewise monotone}. $P$ is also a turning platform for $f$ on $I$, but then $x_P\in J\sq I\sm K'$, contradiction.

\m

(a) and (b) follow easily from the observation before the claim and the claim itself.
\end{proof}

\begin{prop}\label{p:platforms.of.derivatives}
Let $I$ be a nontrivial interval of $\R$, and let $f\colon I\to\R$ be a $C^1$ function. If $P$ is a platform of positive length for $f$, then $P$ is a also a platform for $Df$.
\end{prop}

\begin{proof}
On $P$ we have $Df=0$, so $P$ is at least contained in a platform $Q$ for $Df$. If $a=\min P$ exists in $I$ and $I\cap(-\infty,a)\not=\e$, then arbitrarily close to $a$ we have points $x<a$ where $f(x)\not=f(a)$ and hence there is a point $c_x\in(x,a)$ where $Df(c_x)=(f(x)-f(a))/(x-a)\not=0$ and hence $c_x$ is not on the same platform of $Df$ as $a$. Thus, $Q$ has no points in $(-\infty,a)$. Similarly, it has no points above $\max P$ if it exists. Thus, $P=Q$ is a platform of $Df$.
\end{proof}

\begin{cor}\label{c:platforms.of.derivatives}
Let $n$ be a nonnegative integer. Let $I$ be a nontrivial interval of $\R$, and let $f\colon I\to\R$ be a $C^n$ function. Let $p\in I$. Write $P_k$ for the platform of $D^kf$ containing $p$. Then $P_0\sq\dots\sq P_n$. More specifically, for some initial segment $A\sq\{0,\dots,n\}$ and for some closed interval $P$ of $I$ of positive length, we have for $k=0,\dots,n$ that $P_k=\{p\}$ for $k\in A$ and $P_k=P$ for $k\not\in A$.
\end{cor}

Here $A$ can take the extreme values $A=\e$, $A=\{0,\dots,n\}$. In the latter case $P$ is irrelevant since the corollary says nothing about it.

\begin{prop}\label{p:piecewise.monotone.characterization}
Let $I$ be a nontrivial interval of $\R$, and let $f\colon I\to\R$ be continuous and piecewise monotone. 
\begin{enumerate}
\item
The collection of the turning platforms of $f$ is discrete in $I$.

\item
The collection of the platforms of points in the zero set $Z_f$ of $f$ is discrete in $I$.
\end{enumerate}
\end{prop}

Here, `discrete' means `has no accumulation points'.

\begin{proof}
Let $K$ be a closed discrete set in $I$ which is a witness to the piecewise monotonicity of $f$.

(1) If the turning platforms accumulate at a point $p$, we can find a sequence of them $\{P_n\}$ converging monotonically to $p$, say increasing to $p$. Then $p$ is not the minimum element of $I$ and there is a component $J$ of $I\sm K$ such that either $p\in J$ or $p$ is the right endpoint of $J$. But then for some large enough $n$, $P_n\sq J^\circ$ (the interior of $J$ in $\R$), and hence $f$ is not monotone on $J$, a contradiction since $K$ is a witness to the piecewise monotonicity of $f$.

(2) If the platforms of points in $Z_f$ accumulate at a point $p$, then by the argument in (1), we can find two distinct ones $P_m$ and $P_n$ both contained in $J^\circ$ for some component $J$ of $I\sm K$. But then $f$ is not monotone on $J$ since it is nonzero at some of the points between $P_m$ and $P_n$.
\end{proof}

\begin{prop}\label{p:non.zero.point}
Let $I$ be a nontrivial interval of $\R$, and let $f\colon I\to\R$ be a $C^n$ function. If $p\in I$ and $D^nf(p)\not=0$, then for any $\ep>0$, there is a point $q\in I$ such that $|q-p|<\ep$ and $D^j(q)\not=0$, $j=0,\dots,n$.
\end{prop}

\begin{proof}
Recursively get nonempty open intervals of $\R$, $I_{j}$ for $j=0,\dots,n$, so that $(p-\ep,p+\ep)\cap I\supseteq I_n\supseteq I_{n-1}\supseteq \dots\supseteq I_0$, and $D^jf(x)\not=0$ for all $x\in I_j$. Since $D^nf$ is continuous, there is a nonvoid open interval $I_n\sq I\cap (p-\ep,p+\ep)$ such that $D^nf\not=0$ on $I_n$. Given $I_{j+1}$ with $D^{j+1}f\not=0$ on $I_{j+1}$, we have that $D^{j}f$ is not constant on $I_{j+1}$, so there is a point of $I_{j+1}$ where $D^{j}f$ is nonzero, and then by continuity of $D^{j}f$ there is a nonvoid open interval $I_{j}\sq I_{j+1}$ such that $D^{j}f\not=0$ on $I_{j}$. Let $q$ be any element of $I_0$.
\end{proof}

\begin{prop}[\cite{Bu2019}, Proposition 2.10] \label{p:f.not.piecewise.monotone.implies.Df.also}
Let $I$ be a nontrivial interval of $\R$, and let $f\colon I\to\R$ be a $C^1$ function. If $Df$ is piecewise monotone, then so is $f$. Hence, when $f$ is $C^n$, the set of $k=0,\dots,n$ such that $D^kf$ is piecewise monotone is an initial segment of $\{0,\dots,n\}$.
\end{prop}

In \cite{Bu2019} the following two comonotone approximation and interpolation theorems were obtained. The first deals with functions having no flat points.

\begin{thm}[\cite{Bu2019}, Theorem A]\label{t:piecewise.monotone.interpolation}
Let $n$ be a nonnegative integer, $m$ an integer or $\infty$, with $n\leq m$. Let $U_0\sq U_1\sq\dots$ be a cover of $\R$ by open sets. Suppose $f\colon \mathbb{R}\to\mathbb{R}$ is a $C^m$ function such that $D^nf$ has no flat points.
For any closed discrete set $T$, and for any positive continuous function $\varepsilon\colon\R\to\R$, there is an entire function $g$ such that $g(\R)\sq\R$ and the following hold.
\begin{enumerate}
\item\label{i:spmi:1}
$|D^ig(x) - D^if(x)|<\varepsilon(x)$, $x\in\R\sm U_i$, $0\leq i\leq m$, $i\in\Z$.

\item\label{i:spmi:2}
$D^ig(x) = D^if(x)$, $x\in T\sm U_i$, $0\leq i\leq m$, $i\in\Z$.

\item\label{i:spmi:3}
$D^kg(x)$ has the same sign as $D^kf(x)$, $x\in \R$, $k=0,\dots,n$.
\end{enumerate}
\end{thm}

Since $g$ is analytic, it follows from \ref{i:spmi:3} that each $D^kf$, $k=0,\dots,n-1$, is piecewise monotone and comonotone with $D^kg$ using the union of the zero sets of $D^jf$, $j=0,\dots,n$, as the witnessing set. (Cf.\ Propositions \ref{p:summarize} and \ref{p:no.flat.points.comonotone}.)

The second theorem, Theorem D of \cite{Bu2019} stated here as Theorem \ref{t:C.infty.Hoischen.C.infty.fcts}, allows flat points, but gives comonotone approximation only for $f$, not for its derivatives. In this theorem and also in Theorem \ref{t:C.n.Hoischen.C.n.fcts}, we use the following notation regarding the platforms of a continuous function $f\colon\R \to\R$. If $f$ has a rightmost platform, we denote it $P^f_{\rm max\phantom{i}}$, or just $P_{\max}$. Otherwise we set $P_{\max}=\e$.
If $f$ has a leftmost platform, we denote it $P^f_{\min}$, or just $P_{\min}$. Otherwise we set $P_{\min}=\e$. $P_{\min}$ and $P_{\max}$ coincide when $f$ is constant and otherwise are disjoint.
For a $C^n$ function $f\colon\R\to\R$, we write for simplicity
\[
P^{n}_{\min} = P^{D^nf}_{\min},\ \ P^{n}_{\max} = P^{D^nf}_{\rm max\phantom{i}}.
\]
In Theorems \ref{t:C.infty.Hoischen.C.infty.fcts} and \ref{t:C.n.Hoischen.C.n.fcts}, there is a given closed discrete set $E$. We let $e_{\min}=\inf E$, $e_{\max}=\sup E$, so that $e_{\min}\in E$ if and only if $E$ has a least element (and $e_{\min}\in \{\infty,-\infty\}$ otherwise), and $e_{\max}\in E$ if and only if $E$ has a largest element (and $e_{\max}\in \{\infty,-\infty\}$ otherwise).
Let $W_n$ be the union of the open left and right half-lines determined by $e_{\min}$ and $e_{\max}$, respectively, when these belong to $P^n_{\min}$ and $P^n_{\max}$, respectively:
\begin{equation}
W_n = \{x\in\R:\ e_{\min}\in P^n_{\min}\ \text{\rm and}\ x<e_{\min},\ \text{\rm or}\ e_{\max}\in P^n_{\max}\ \text{\rm and}\ x>e_{\max}\}.
\end{equation}

The statement of Theorem D in \cite{Bu2019} is for $C^n$ functions where we allow $n=\infty$. In this paper we deal only with the case where $n$ is finite, so we restrict the statement to that case.

\begin{thm}[\cite{Bu2019}, Theorem D]\label{t:C.infty.Hoischen.C.infty.fcts}
Let $n$ be a nonnegative integer. Let $f\colon \R\to\R$ be a nonconstant piecewise monotone $C^n$ function, and let $E\sq \R$ be a closed discrete set which has no more than one point on each platform of $f$ and has exactly one point on each turning platform. Let $\ep\colon \R\to\R$ be a positive continuous function and let $\ep_0>0$.
Then there is a function $g\colon\R\to\R$ which is the restriction to $\R$ of an entire function and satisfies the following conditions.
\begin{enumerate}
\item\label{i:t:C.infty.Hoischen:5}
$D^jf(x)=D^jg(x)$ when $x\in E$, $0\leq j\leq n$.

\item\label{i:t:C.infty.Hoischen:1}
$Dg(x)\not=0$ when $x\notin E$ and, if $n=0$, also when $x\in E$ is not a turning point of $f$.

\item\label{i:t:C.infty.Hoischen:2}
$|g(x)-f(x)|<\ep_0$ for all $x\in \R$.

\item\label{i:t:C.infty.Hoischen:3}
$|g(x)-f(x)|<\ep(x)$ when $x\notin W_0$.

\item\label{i:t:C.infty.Hoischen:4}
$|D^jg(x)-D^jf(x)|<\ep(x)$ when $x\in \R$, $1\leq j\leq n$.
\end{enumerate}
\end{thm}

As long as $\ep(x)$ is small enough, these properties force $g$ to be comonotone with $f$ with $E$ as witnessing set. To see this, note that on a component $I$ of the complement of $E$, $g$ is monotone by \ref{i:t:C.infty.Hoischen:1}, and $f$ is monotone since $E$ has a point on each turning platform. That $g$ and $f$ have the same monotonicity on $I$ is clear if $f$ is constant, and otherwise it follows from the condition $|g(x)-f(x)|<\ep(x)$ in \ref{i:t:C.infty.Hoischen:3} as long as $\ep(x)$ is taken small enough.

In this paper, we prove the following theorems.
The first one provides an approximation to a polynomial on a half-line having spedified derivatives of order $\geq n$ at the endpoint.

\begin{thm*}[\ref{t:approx.a.poly.on.platform}]
Let $f$ be a polynomial with $D^nf$ having constant value $k$, for some nonnegative integer $n$. Let $r,s\in\{-1,1\}$. Let $a\in\R$, $I=[a,\infty)$.
Let $\ep(x)$ be a positive continuous function on $rI$. Fix a positive number $b \leq \ep(ra)$, with $b<\ep(ra)$ unless $\ep(x)=\ep_0$ is constant. Let $\beta_j$, $j=0,1,\dots$ be a sequence of real numbers such that $\beta_0=k-sr b$ and either $\beta_j=0$ for all $j>0$, or for the least $j>0$ with $\beta_j\not=0$ we have $sr^{j+1}\beta_j > 0$.
Then there is a $C^\infty$ function $g\colon rI\to\R$ such that
\begin{enumerate}
\item
$s D^{n+1}g(x)>0$, $x\in rI$, $x\not=ra$.

\item
$D^{n+j}g(ra)=\beta_j$, $j=0,1,\dots$

\item
$|D^jg(x)-D^jf(x)|<\ep(x)$, $x\in rI$, $x\not=ra$, $j=0,\dots,n$.
\end{enumerate}
Now assume that $\lim_{x\to r\infty}\ep(x)=0$. From {\rm(1)} and {\rm(3)} we get {\rm(4)} and {\rm(5)} for $j=0,\dots,n$.
\begin{enumerate}
\setcounter{enumi}{3}
\item
$s (-r)^{n+j+1}D^{j}g(x) > s(-r)^{n+j+1}D^{j}f(x)$, for all $x\in rI$.

\item
If $1\leq j\leq n$, we have $D^{j-1}g(ra)-D^{j-1}f(ra) = -r\int_{rI} (D^{j}g-D^{j}f)$.
\end{enumerate}
If $D^jf$, $j=0,\dots,n$, are all monotone on $rI$, and those which are not constant are nonzero at $ra$, then we can ask moreover that
\begin{enumerate}
\setcounter{enumi}{5}
\item
$D^jg(x)\not=0$, $x\in rI$, $j=1,\dots,n$.
\end{enumerate}
\end{thm*}

The second theorem also provides an approximation to a polynomial $f$ on a half-line whose derivatives at the endpoint agree with those of $f$ up to order $n$ and have specified values for orders larger than $n$. The approximation is only for the $n$th derivative, but when the derivatives up to order $n$ for $f$ are all monotone, the derivatives of order up to $n$ for the approximating function have the same monotonicity as those of $f$.

\begin{thm*}[\ref{t:approx.a.poly.on.platform}]
Let $f$ be a polynomial with $D^nf$ having constant value $k$, for some nonnegative integer $n$. Let $r,s\in\{-1,1\}$. Let $a\in\R$, $I=[a,\infty)$.
Let $\ep>0$. Let $\beta_j$, $j=1,2,\dots$ be a sequence of real numbers such either $\beta_j=0$ for all $j=1,2,\dots$, or for the least $j$ with $\beta_j\not=0$ we have $sr^{j+1}\beta_j > 0$.
Then there is a $C^\infty$ function $h\colon rI\to\R$ such that
\begin{enumerate}
\item\label{c.1}
$s D^{n+1}h(x)>0$, $x\in rI$, $x\not=ra$.

\item\label{c.2}
$D^{n+j}h(ra)=\beta_j$, $j=1,2,\dots$

\item\label{c.3}
$D^jh(ra) = D^jf(ra)$, $j=0,\dots,n$.

\item\label{c.4}
$|D^nh(x)-D^nf(x)| = |D^nh(x)-k|<\ep$, $x\in rI$.
\end{enumerate}
If $D^jf$ is monotone on $rI$ for each $j=0,\dots,n$ then there is a choice of $s\in\{1,-1\}$ such that for any $C^\infty$ function $h\colon rI\to\R$ satisfying \ref{c.1} and \ref{c.3} also satisfies for each $j=0,\dots,n$,
\begin{enumerate}
\setcounter{enumi}{4}
\item\label{c.5}
$D^{j+1}h(x)\not=0$ for all $x\in rI\sm\{ra\}$.

\item\label{c.6}
$D^jh$ has the same monotonicity as $D^jf$ on $rI$.
\end{enumerate}
\end{thm*}

The next theorem improves Theorem \ref{t:C.infty.Hoischen.C.infty.fcts} by providing comonotonicity of $D^jg$ with $D^jf$ for all $j=0,\dots,n$. It assumes a statement $(Q_n)$ which is discussed in the next section. This statement was shown in \cite{BHM2025} to hold for $n\leq 3$. (In \cite{BHM2025} it was called $(P_n)$, but we call it $(Q_n)$ here to avoid confusion with our notation for platforms.)

\begin{thm*}[\ref{t:C.n.Hoischen.C.n.fcts}]
Assume $(Q_n)$. Let $n$ be a nonnegative integer. Let $f\colon \R\to\R$ be a $C^n$ function such that $D^nf$ is piecewise monotone and nonconstant. Let $E\sq \R$ be a closed discrete set which has no more than one point on any platform of $D^nf$, and has a point on each turning platform of $D^jf$, $j=0,\dots,n$. Let $\ep\colon \R\to\R$ be a positive continuous function and let $\ep_0>0$.
Then there is a function $g\colon \R\to\R$ which is the restriction of an entire function and satisfies the following conditions for $x\in \R$.
\begin{enumerate}
\item\label{p:C.n.Hoischen:5}
$D^jg(x)=D^jf(x)$ when $x\in E$, $j=0,\dots,n$.

\item\label{p:C.n.Hoischen:1a}
$D^jg(x)\not=0$ when $x\not\in E$, $1\leq j\leq n+1$.

\item\label{p:C.n.Hoischen:2}
$D^{n+1}g(x)\not=0$ when $x\in E\sm(P^n_{\min}\cup P^n_{\max})$ and $x$ not a turning point of $D^nf$

\item\label{p:C.n.Hoischen:3}
$|D^ng(x)-D^nf(x)|<\ep_0$.

\item\label{p:C.n.Hoischen:4}
$|D^jg(x)-D^jf(x)|<\ep(x)$, when $x\notin W_n$, $j=0,\dots,n$.

\item\label{p:C.n.Hoischen:6}
$D^jg$ is comonotone with $D^jf$, with $E$ as witnessing set, $j=0,\dots,n$.
\end{enumerate}
\end{thm*}

In this paper,
the one-sided derivatives of a function defined on a nontrivial interval at endpoints of that interval will be denoted by the same notation as the two-sided derivatives, leaving it for the context to distinguish the two.

We record for later reference the following simple observations.

\begin{prop}\label{p:monotone.der.sign}
Let $n$ be a nonnegative integer, and let $f$ be a polynomial whose derivatives $D^jf$, $j=0,\dots,n$, are all monotone on a nontrivial interval $I$.
\begin{enumerate}
\item
Those derivatives $D^jf$, $j=0,\dots,n$, which are not constant are strictly monotone on $I$ and those derivatives $D^jf$, $j=1,\dots,n$, which are not identically zero have no zeros on $I$ except possibly at endpoints of $I$.

\item
If $I=rJ$, where $J=[a,\infty)$ and $r\in\{1,-1\}$, $f$ is nonzero of degree $m\leq n$, and the coefficient of the term of highest degree $m$ in $f$ has sign $t\in\{1,-1\}$, then for $j=1,\dots,m$, $D^jf$ has constant sign $r^{m-j}t$ on $I\sm\{ra\}$.
\end{enumerate}
\end{prop}

\begin{proof}
(1) This is easy if $n=0$ and vacuous if $f=0$, so assume $n\geq 1$ and $f$ is nonzero with degree $m$. $D^mf$ is a nonzero constant, The other nonzero derivatives $D^jf$, $j=0,\dots,n$, are those for which  $0\leq j\leq \min(m-1,n)$. They are strictly monotone on $I$ since $f$ is a polynomial. For $1\leq j\leq \min(m-1,n)$, $D^jf$ cannot have a zero except at $a$, because otherwise it changes sign on $I$ and hence $D^{j-1}f$ is not monotone on $I$.

(2) The derivatives $D^jf$, $j=1,\dots,m$, all have their coefficient of the term of highest degree of sign $t$. When $|x|$ is large, the sign of $D^jf(x)$ is $r^{m-j}t$, so, since the sign is constant by (1), $D^jf(x)$ has sign $r^{m-j}t$ for all $x\in I\sm\{ra\}$.
\end{proof}

\section{The statement $(Q_n)$}
\label{s:Final values of increasing $C^n$ functions and their derivatives}

In this section, we explain the statement $(Q_n)$ mentioned in the introduction and indicate its usefulness for obtaining approximations.
We first introduce some notation.

For $a=(a_0,\dots,a_n)$ and $b=(b_0,\dots,b_n)$ in $\R^{n+1}$, let $(a;b)$ denote the concatenation $(a;b) = (a_0,\dots,a_n,b_0,\dots,b_n)\in \R^{2(n+1)}$.
We also write this tuple as $(a_j\,;\,b_j)$ when $n$ is clear from the context.
Let $c<d$ be real numbers. Let $\mathcal{S}$ denote the set of infinite sequences $\al=(\al_0,\al_1,\al_2,\dots)$ of real numbers, and let $x=(x_0,\dots,x_n)\in\R^{n+1}$. For $s\in\{1,-1\}$, we make the following definitions.

\begin{align*}
\F_n^s[c,d] & = \{f\in C^n[c,d]: sD^nf\ \text{is increasing but not constant}\}, \\
\F_n^{s,\infty}[c,d] & = \{f\in C^\infty[c,d]: sD^{n+1}f(x)>0\ \text{for all}\ x\in (c,d)\}, \\
\A_n^s(x) & = \{\al\in\mathcal{S}: \al_j=x_j,\, j=0,\dots,n,\ \text{and either}\ \al_j=0\ \text{for all}\ j>n, \\
& \rule{2.02cm}{0cm}\text{or for the least $j>n$ such that $\al_j\not=0$, we have $s\al_j>0$}\}, \\
\B_n^s(x) & = \{\beta\in\mathcal{S}: \beta_j=x_j,\, j=0,\dots,n,\ \text{and either}\ \beta_j=0\ \text{for all}\ j>n, \\
& \rule{2.02cm}{0cm}\text{or for the least $j>n$ such that $\beta_j\not=0$, we have $s(-1)^{n+j+1}\beta_j>0$}\}, \\
V_n^s[c,d] & = \{(a;b): a,b\in\R^{n+1}\ \text{and there is an $f\in \F_n^s[c,d]$} \\
& \rule{1.9cm}{0cm}\text{such that $D^jf(c)=a_j$ and $D^jf(d)=b_j$ for all $j=0,\dots,n$}\}, \\
V_n^{s,\infty}[c,d] & = \{(a;b): a,b\in\R^{n+1}\ \text{and for all $\al\in\A_n^s(a)$, $\beta\in\B_n^s(b)$ there is an $f\in \F_n^{s,\infty}[c,d]$} \\
& \rule{1.9cm}{0cm}\text{such that $D^jf(c)=\al_j$ and $D^jf(d)=\beta_j$ for all $j=0,1,2,\dots$}\}.
\end{align*}

When $s=1$ we omit it in the notation, writing $\F_n[c,d]$, $\F_n^\infty[c,d]$, $\A_n(x)$, and so on.
Note that for $s\in\{1,-1\}$,
\begin{itemize}
\item
$\F_n^{s}[c,d] = s\F_n[c,d]$ and $\F_n^{s,\infty}[c,d] = s\F_n^{\infty}[c,d]$,

\item
$\A_n^{s}(x) = s\A_n(sx)$ and $\B_n^{s}(x) = s\B_n(sx)$,

\item
$V_n^{s}[c,d] = sV_n[c,d]$ and $V_n^{s,\infty}[c,d] = sV_n^{\infty}[c,d]$.
\end{itemize}

The following remark explains the restrictions on the sequences in definitions of $\A^s_n(x)$ and $\B^s_n(x)$.

\begin{rem}\label{r:signs}
If $f\in\F_n^{s,\infty}[c,d]$ has $D^jf(c)=\al_j$ and $D^jf(d)=\beta_j$ for all nonnegative integers $j$, then the least $j>n$ for which $\al_j$ is nonzero, if there is one, satisfies $s\al_j>0$, and the least $j>n$ for which $\beta_j$ is nonzero, if there is one, satisfies that $s\beta_j$ is positive if $n+j$ is odd and negative if $n+j$ is even (\cite{BHM2025}, Remark 4.1).
\end{rem}

The next remark will not be used in the sequel.

\begin{rem}\label{r:cones}
(a) The families $\F_n^{s}[c,d]$ and $s\F_n^{\infty}[c,d]$
are both convex cones (in $C^n[c,d]$ and in $C^\infty[c,d]$, respectively) in the sense of \cite[\S 27]{Be1974}, i.e., they are closed under taking linear combinations with positive coefficients.

(b) $V_n^s[c,d]$ and $V_n^{s,\infty}[c,d]$ are convex cones in $\R^{2(n+1)}$.
(Proof.
We may assume $s=1$. $V_n[c,d]$ is easily seen to be closed under positive scaling and under sums, by scaling or adding the witnessing functions, respectively. For $V_n^{\infty}[c,d]$ we can proceed as follows.

For closure under positive scaling, suppose $(a;b)\in V_n^{\infty}[c,d]$ and $\lb>0$. Given $\al\in\A_n(\lb a)$, $\beta\in\B_n(\lb b)$, we have $\lb^{-1}\al\in\A_n(a)$, $\lb^{-1}\beta\in\B_n(b)$, so there is an $f\in \F_n^{\infty}[c,d]$ such that $D^jf(c)=\lb^{-1}\al_j$ and $D^jf(d)=\lb^{-1}\beta_j$ for all $j=0,1,2,\dots$. Then $\lb f\in \F_n^{\infty}[c,d]$ and $D^j(\lb f)(c)=\al_j$, $D^j(\lb f)(d)=\beta_j$ for all $j=0,1,2,\dots$, showing that $(\lb a;\lb b)\in V_n^{\infty}[c,d]$.

For closure under sums, suppose $(a^i;b^i)\in V_n^{\infty}[c,d]$, $i=1,2$. Given $\al\in\A_n(a^1+a^2)$, $\beta\in\B_n(b^1+b^2)$, write $\al=\al^1+\al^2$, where $\al^i_j=a^i_j$, $j=0,\dots,n$, $i=1,2$, and $\al^1_j=\al_j$, $\al^2_j=0$, $j>n$. Similarly, write $\beta=\beta^1+\beta^2$, where $\beta^i_j=b^i_j$, $j=0,\dots,n$, $i=1,2$, and $\beta^1_j=\beta_j$, $\beta^2_j=0$, $j>n$.
Then $\al^i\in\A_n(a^i)$, $\beta^i\in\B_n(b^i)$, $i=1,2$, so there is are $f_i\in \F_n^{\infty}[c,d]$ such that $D^jf_i(c)=\al^i_j$ and $D^jf_i(d)=\beta^i_j$ for all $j=0,1,2,\dots$. Then $f=f_1+f_2\in \F^{\infty}_n[c,d]$ and $D^j(f)(c)=\al_j$, $D^j(f)(d)=\beta_j$ for all $j=0,1,2,\dots$, showing that $(a^1;b^1)+(a^2;b^2)\in V_n^{\infty}[c,d]$.)
\end{rem}

We state for emphasis the following simple but useful fact.

\begin{rem}\label{r:V_n}
If $(a;b)\in V_n^s[c,d]$ then $sa_n < sb_n$. (Given a witnessing function $f\in \F_n^s[c,d]$, we have $sa_n=sD^nf(c) < sD^nf(d)=sb_n$, where the strict inequality holds since $sD^nf$ is increasing but not constant.)
\end{rem}

We are now ready to state $(Q_n)$. We also define a weaker statement $(Q_n^-)$.

\begin{defn}
$(Q_n)$ is the statement that for all real numbers $c<d$ and $s\in\{1,-1\}$,
\[
V_n^s[c,d] = V_n^{s,\infty}[c,d]\ \text{and}\ V_n^s[c,d]\ \text{is open in $\R^{2(n+1)}$}.
\]
$(Q_n^-)$ is the statement that for all real numbers $c<d$ and $s\in\{1,-1\}$,
\[
V_n^s[c,d] = V_n^{s,\infty}[c,d].
\]
\end{defn}

\begin{conj}[\cite{BHM2025}]
$(Q_n)$ is true for all nonnegative integers $n$.
\end{conj}

The following was established in \cite{BHM2025}.

\begin{thm}[\cite{BHM2025}, Theorem A]
$(Q_n)$ holds for $n\leq 3$.
\end{thm}

We now indicate how $(Q_n)$ is useful for obtaining monotone approximations. The proof makes use of the following simple consequence of uniform continuity.

\begin{prop}[\cite{Bu2019}, Proposition 6.1]\label{p:small.intervals.f.not.constant}
Let $f\colon[a,b]\to\R$ be a nonconstant continuous function and let $\ep>0$. Then there is a partition $a=x_0<\dots<x_n=b$ of the interval $[a,b]$ so that $n\geq 2$ and on each subinterval $[x_i,x_{i+1}]$ the range of $f$ has diameter less than $\ep$ but $f$ is not constant.
\end{prop}

The statement in \cite{Bu2019} did not require $n\geq 2$, but if $n=1$ we can choose any point $x'\in (a,b)$ where $f(x')\notin\{f(a),f(b)\}$ and use the partition $\{x_0,x_1,x_2\} = \{a,x',b\}$.

The next theorem extends Proposition 6.3 of \cite{Bu2019} which deals with the case $n=0$. The argument for parts (1)--(3) is building on the proof of the corresponding parts of that proposition. The additional parameter $r$ allows us to phrase the properties of $f(x)$ in terms of those of $f(-x)$ and this formulation is useful later.

\begin{prop}\label{p:approx.n.der.incr}
Assume $(Q_n)$. Let $a<b$, $I=[a,b]$, $\ep>0$, $r,s\in\{1,-1\}$. Let $f\colon rI\to\R$ be a $C^{n}$ function with $sD^nf$ nonconstant and increasing on $rI$. There is a $\de>0$ such that the following holds. Let $\al_0,\al_1,\dots$ and $\beta_0,\beta_1,\dots$ be sequences of real numbers such that
\begin{enumerate}[$\bullet$]
\item
$|\al_j-D^jf(ra)|<\de$ and $|\beta_j-D^jf(rb)|<\de$, $j=0,\dots,n$;

\item
either $\al_j=0$ for all $j>n$, or for the least $j>n$ for which $\al_j\not=0$ we have $sr^{n+j+1}\al_j>0$;

\item
either $\beta_j=0$ for all $j>n$, or for the least $j>n$ for which $\beta_j\not=0$ we have $s(-r)^{n+j+1}\beta_j>0$.
\end{enumerate}
Then 
there is a $C^\infty$ function $g$ such that
\begin{enumerate}
\item
$D^jg(ra)=\al_j$ and $D^jg(rb)=\beta_j$, $j=0,1,\dots$

\item
$sD^{n+1}g(rx) > 0$, $a<x<b$

\item
$|D^jg(rx)-D^jf(rx)|<\ep$, $a\leq x\leq b$, $j=0,\dots,n$

\item
For $c=a,b$, $D^jg(rc)$ has the same sign as $D^jf(rc)$ if $D^jf(rc)\not=0$, $j=0,\dots,n$.
\end{enumerate}
Suppose $D^jf$, $j=0,\dots,n$, are all monotone on $rI$. Then for $j=1,\dots,n$, $D^jf(ra)$ and $D^jf(rb)$ are not of opposite sign.
Under the additional assumption that for each $j=1,\dots,n$,
\begin{enumerate}[\rm(a),leftmargin=1.3cm]
\item
if $D^jf(ra)=0$ then $\al_jD^jf(rb)\geq 0$, and

\item
if $D^jf(rb)=0$ then $\beta_jD^jf(ra)\geq 0$,
\end{enumerate}
we can require that $D^jg(ra)$ has the same sign as $D^jf(ra)$ when $D^jf(ra)\not=0$, $D^jg(rb)$ has the same sign as $D^jf(rb)$ when $D^jf(rb)\not=0$, and
\begin{enumerate}
\setcounter{enumi}{4}
\item\label{2.9.5}
$D^jg(rx)\not=0$, $j=1,\dots,n$, $a<x<b$
\end{enumerate}
\end{prop}

\begin{proof}
Note that $sD^nf$ being nonconstant and increasing on $rI$ is equivalent to $x\mapsto srD^nf(rx)$ being nonconstant and increasing on $I$. (This is clear if $r=1$, and if $r=-1$ then $x\mapsto srD^nf(rx) = -sD^nf(-x)$ has the same monotonicity as $sD^nf$.)

We may assume $r=s=1$. (Given this case and $I=[a,b]$, $r$, $s$, $f$ as in the assumption, define $\bar{f}\colon I\to\R$ by $\bar{f}(x)=sr^{n+1}f(rx)$. Then $D^n\bar{f}(x) = srD^nf(rx)$ is nonconstant and increasing. By the case $r=s=1$, we have a $\de>0$ satisfying the conclusion of the proposition for $\bar{f}$. It is straightforward to verify that this same $\de$ satisfies the conclusion for $f$. Briefly, if sequences $\al_j$, $\beta_j$ satisfy the three bullets, check that $\bar{\al}_j=sr^{n+j+1}\al_j$, $\bar{\beta}_j=sr^{n+j+1}\beta_j$ satisfy those bullets for $\bar{f}$. By the choice of $\de$, get a $C^\infty$ function $\bar{g}\colon I\to\R$ satisfying (1)--(5) with respect to $\bar{f}$ and $\bar{\al}_j,\bar{\beta}_j$, and take $g(x) = sr^{n+1}\bar{g}(rx)$.)

From Proposition \ref{p:small.intervals.f.not.constant}, get points $a=x_0<\dots<x_{k}=b$ so that $k\geq 2$ and for $i=0,\dots,k-1$,
\begin{enumerate}[(i)]
\item\label{i:g0.Hoischen:i}
$D^nf(x_i)<D^nf(x_{i+1})$

\item\label{i:g0.Hoischen:ii}
$(b-a)^t(D^nf(x_{i+1})-D^nf(x_{i}))<\ep$, for $t=0,\dots,n$
\end{enumerate}
(Only the values $t=0,n$ need be mentioned in (ii), but the above formulation is convenient.)
Since $V_n[x_0,x_1]$ and $V_n[x_{k-1},x_k]$ are open, there is a $\de>0$ such that
\begin{enumerate}
\item[(iii)]
$|\al_j-D^jf(x_0)|<\de$ implies
$(\al_j\,;\,D^jf(x_1))\in V_n[x_0,x_1]$, $j=0,\dots,n$.

\item[(iv)]
$|\beta_j-D^jf(x_{k})|<\de$ implies
$(D^jf(x_{k-1})\,;\,\beta_j)\in V_n[x_{k-1},x_k]$, $j=0,\dots,n$.
\end{enumerate}
Furthermore, take $\de$ small enough so that (ii) holds in the form
\begin{enumerate}
\renewcommand{\theenumi}{\alph{enumi}}
\item[(ii)$'$]
$(b-a)^t(D^nf(x_{i+1})-D^nf(x_{i})+\de)<\ep$, for $t=0,\dots,n$
\end{enumerate}
Now suppose that $\al_j$, $\beta_j$ are sequences as given in the hypothesis for this value of $\de$. Define sequences $(\al^i_j:j=0,1,\dots)$, $i=0,\dots,k$, as follows.
\begin{itemize}
\item
$\al^0_j=\al_j$, $\al^{k}_j=\beta_j$

\item
for $0<i<k$, $\al^i_j=D^jf(x_i)$, $j=0,\dots,n$, $\al^i_{n+1}=1$, $\al^i_j=0$, $j>n+1$
\end{itemize}
We have
\[
\al^0_n<\al^1_n<\dots<\al^{k-1}_n<\al^k_n.
\]
Indeed, for $i=1,\dots,k-2$, $\al^i_n = D^nf(x_i) < D^nf(x_{i+1})=\al^{i+1}_n$ by (i).  We also have $\al^0_n = \al_n < D^nf(x_1) = \al^1_n$ by (iii) and Remark \ref{r:V_n}. Similarly, $D^nf(x_{k-1}) < \beta_n = \al^k_n$.
By $(Q_n)$ there are $C^\infty$ functions $g_i\colon[x_i,x_{i+1}]\to\R$, $i=0,\dots,k-1$, such that
\begin{enumerate}[(i)]
\setcounter{enumi}{4}
\item
$D^jg_i(x_i)=\al^i_j$ and $D^jg_i(x_{i+1})=\al^{i+1}_j$, $j=0,1,\dots$

\item
$D^{n+1}g_i(x)>0$, $x_i<x<x_{i+1}$

\item
$|D^ng_i(x)-D^nf(x)|<\ep/(b-a)^t$, $t=0,\dots,n$, $x_i\leq x\leq x_{i+1}$
\end{enumerate}
For (vii), for $i=1,\dots,k-2$, use the fact that, by (ii), $D^nf$ varies by less than $\ep/(b-a)^t$ on $[x_i,x_{i+1}]$, and the fact that $D^nf$ and $D^ng_i$ are both increasing on $[x_i,x_{i+1}]$ and have the same values at the endpoints. On $[x_0,x_1]$, the functions $D^nf$, $D^ng_0$ agree at the right endpoint, so their difference at any point is not more that the larger of $|D^nf(x_1)-D^nf(x_0)|<\ep/(b-a)^t$ and (using (ii)$'$)
\begin{align*}
|D^ng(x_1)-D^ng(x_0)| & = |D^nf(x_1)-\al_n| \\
& \leq |D^nf(x_1)-D^nf(x_0)|+|D^nf(x_0)-\al_n| \\
& < |D^nf(x_1)-D^nf(x_0)|+\de<\ep/(b-a)^t.
\end{align*}
Similarly for $i=k-1$.
Now we verify, for $x\in[x_i,x_{i+1}]$, by induction on $j=0,\dots,n$ that
\begin{enumerate}
\item[(viii)]
$|D^{n-j}g_i(x) - D^{n-j}f(x)| < \ep/(b-a)^{t-j}$, $t=0,\dots,n$.
\end{enumerate}
Taking $t=j$ will then give us
\[
|D^{n-j}g_i(x) - D^{n-j}f(x)| < \ep,\ j=0,\dots,n,
\]
and then $g=g_0\cup\dots\cup g_{k-1}$ satisfies (1)--(3).

The case $j=0$ of (viii) is (vii).
Suppose $0<j\leq n$ and (viii) holds for $j-1$. Consider first the case $0<i<k$. Since $g_i$ and $f$ have the same derivatives up to order $n$ at $x_i$, we have
\begin{align*}
& |D^{n-j}g_i(x) - D^{n-j}f(x)| \\
& = \left|D^{n-j}g_i(x_i) + \int_{x_i}^x D^{n-j+1}g_i(t)\,dt - D^{n-j}f(x_i)-\int_{x_i}^x D^{n-j+1}f(t)\,dt \right| \\
& \leq \int_{x_i}^x |D^{n-j+1}g_i(t) - D^{n-j+1}f(t)|\,dt \\
& \leq (x-x_i)\ep/(b-a)^{t-j+1} \leq \ep/(b-a)^{t-j}
\end{align*}
and the second inequality is strict unless $x=x_i$, in which case the third inequality is strict. If $i=0$, a similar argument works replacing $x_i$ by
$x_{i+1}=x_1$ and replacing the final $\int_{x_i}^x$ by $\int_x^{x_1}$.

Clause (4) is immediate from (1) as long as we choose $\de>0$ small enough so that the conditions $|\al_j-D^jf(a)|<\de$ and $|\beta_j-D^jf(b)|<\de$ ensure that for $j=0,\dots,n$,
\begin{itemize}
\item
$\al_j$ has the same sign as $D^jf(a)$ when $D^jf(a)\not=0$, and

\item
$\beta_j$ has the same sign as $D^jf(b)$ when $D^jf(b)\not=0$.
\end{itemize}

For (5), assume that $D^jf$, $j=0,\dots,n$, are all monotone on $[a,b]$. Since $D^nf$ is not constant, neither are any of $D^jf$, $j=0,\dots,n$. Since they are monotone, it follows that $D^jf(a)\not=D^jf(b)$ for $j=0,\dots,n$, and in particular if one of these values is zero, then the other is nonzero.
For $j=1,\dots,n$, the fact that $D^{j-1}f$ is monotone implies that $D^jf$ is everywhere $\leq 0$ or everywhere $\geq 0$. In particular, $D^jf(a)$ and $D^jf(b)$ are not of opposite sign. By (4) and the conditions (a) and (b), it then follows that $\al_j$ and $\beta_j$ are not of opposite sign. (For example, if $D^jf(a)=0$ then as pointed out above, $D^jf(b)\not=0$, and by (a), either $\al_j=0$ or $\al_j$ has the same sign as $D^jf(b)$ and hence has the same sign as $\beta_j$.)
It now follows by reverse induction on $j=1,\dots,n$, starting with $j=n$ and using (2) for this initial step, that $D^jg$ is strictly monotone and hence on $(a,b)$ takes values strictly between $D^jg(a)=\al_j$ and $D^jg(b)=\beta_j$, and hence takes nonzero values on $(a,b)$.
\end{proof}

If $\al_j=D^jf(ra)$, $\beta_j=D^jf(rb)$ for $j=0,\dots,n$, then (a) and (b) are trivially satisfied since their hypotheses are $\al_j=0$ and $\beta_j=0$, respectively. The proof in this case uses only $(Q_n^-)$ and is obtained by removing all references to $\de$ (more precisely, (iii), (iv), (ii)$'$ should be removed, in the explanation of (vii), the Case $[x_0,x_1]$ is the same as that for the other intervals $[x_i,x_{i+1}]$, and clause (4) is trivial, so it and its proof can be deleted). We state this version as a proposition.

\begin{prop}\label{p:approx.n.der.incr.minus}
Assume $(Q_n^-)$. Let $a<b$, $I=[a,b]$, $\ep>0$, $r,s\in\{1,-1\}$. Let $f\colon rI\to\R$ be a $C^{n}$ function with $sD^nf$ nonconstant and increasing on $rI$. Let $\al_0,\al_1,\dots$ and $\beta_0,\beta_1,\dots$ be sequences of real numbers such that
\begin{enumerate}[$\bullet$]
\item
$\al_j=D^jf(ra)$ and $\beta_j=D^jf(rb)$, $j=0,\dots,n$;

\item
either $\al_j=0$ for all $j>n$, or for the least $j>n$ for which $\al_j\not=0$ we have $sr^{n+j+1}\al_j>0$;

\item
either $\beta_j=0$ for all $j>n$, or for the least $j>n$ for which $\beta_j\not=0$ we have $s(-r)^{n+j+1}\beta_j>0$.
\end{enumerate}
Then 
there is a $C^\infty$ function $g$ such that
\begin{enumerate}
\item\label{det.1}
$D^jg(ra)=\al_j$ and $D^jg(rb)=\beta_j$, $j=0,1,\dots$

\item\label{det.2}
$sD^{n+1}g(rx) > 0$, $a<x<b$

\item\label{det.3}
$|D^jg(rx)-D^jf(rx)|<\ep$, $a\leq x\leq b$, $j=0,\dots,n$
\end{enumerate}
Suppose $D^jf$, $j=0,\dots,n$, are all monotone on $rI$. Then for $j=1,\dots,n$, $D^jf(ra)$ and $D^jf(rb)$ are not of opposite sign.
We can require that
\begin{enumerate}
\setcounter{enumi}{3}
\item\label{det.4}
$D^jg(rx)\not=0$, $j=1,\dots,n$, $a<x<b$
\end{enumerate}
\end{prop}

\section{Approximating polynomials on a half-line}
\label{s:Approximating polynomials}

In this section we prove Theorems \ref{t:approx.a.poly.on.platform} and \ref{t:approx.a.poly.on.platform.2} which provide approximations to a polynomial $f$ and its derivatives on a half-line by a $C^\infty$ function $g$ and its derivatives, with the derivatives of $g$ having prescribed values of order larger than $n$ at the endpoint of the half-line.

The first lemma strengthens \cite{Bu2019}, Lemma 6.4, by requiring $\int_x^\infty u < u(x) < \ep(x)$ rather than just $\int_x^\infty u < \ep(x)$.

\begin{lem}\label{l:u.drops.quickly.to.zero}
Let $\ep\colon[0,\infty)\to\R$ be a positive continuous functions.
Let $\beta_j$, $j=0,1,\dots$ be a sequence of real numbers such that $0<\beta_0<\ep(0)$ and either $\beta_j=0$ for all $j>0$, or for the least $j>0$ with $\beta_j\not=0$ we have $\beta_j < 0$.
Then there is a positive $C^\infty$ function $u$ on $[0,\infty)$ such that
\begin{enumerate}
\item
$Du(x)<0$, $x>0$.

\item
$D^ju(0)=\beta_j$, $j=0,1,\dots$

\item
$\int_x^\infty u(t)\,dt < u(x) < \ep(x)$, $x\geq 0$.
\end{enumerate}
If $\ep(x)=\ep_0$ is a constant function, then we can allow $\beta_0=\ep_0$ with the inequality $u(x) < \ep(x)=\ep_0$ holding for $0< x<\infty$.
\end{lem}

\begin{proof}
We may assume that $\ep(x)$ is decreasing (in the non-strict sense that $x\leq y$ implies $\ep(x)\geq \ep(y)$).
Define positive numbers $b_0>b_1>\dots$ and numbers $0=a_0<a_1<\dots$ as follows.
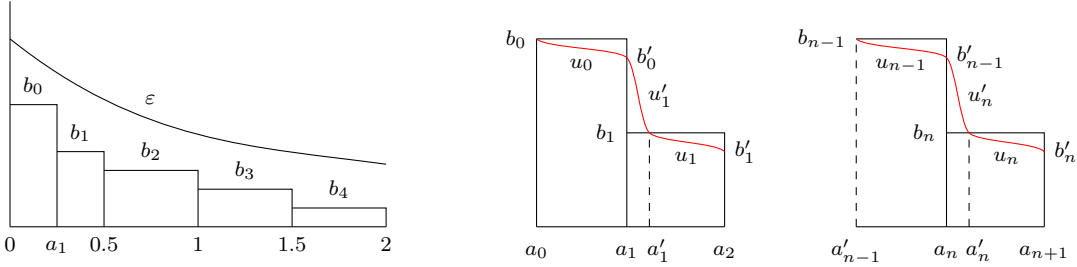
\[
\begin{tikzpicture}
\def\u{2.5cm}
\def\d{2.5cm}
\draw (0,0) -- (2*\u,0);
\draw (0,0) -- (0,1.2*\d);
\foreach \i/\j in {0/0,0.25/a_1,0.5/0.5,1/1,1.5/1.5,2/2}
{
\node[below] at (\i*\u,0) {\sz \rule{0pt}{5pt}$\j$};
}
\draw (0,0.65*\d) -- node[above] {\sz $b_0$} ++(0.25*\u,0) -- ++(0,-0.65*\d);
\draw (0.25*\u,0.4*\d) -- node[above] {\sz $b_1$} ++(0.25*\u,0) -- ++(0,-0.4*\d);
\draw (0.5*\u,0.3*\d) -- node[above] {\sz $b_2$} ++(0.5*\u,0) -- ++(0,-0.3*\d);
\draw (1*\u,0.2*\d) -- node[above] {\sz $b_3$} ++(0.5*\u,0) -- ++(0,-0.2*\d);
\draw (1.5*\u,0.1*\d) -- node[above] {\sz $b_4$} ++(0.5*\u,0) -- ++(0,-0.1*\d);
\coordinate (start) at (0*\u,1*\d);
\coordinate (end) at (2*\u,0.333*\d);
\coordinate (c1) at ($(start) + 1.4*(0.5*\u,-0.4*\d)$);
\coordinate (c4) at ($(end) - 1.4*(0.5*\u,-0.067*\u)$);
\draw (start) .. controls (c1) and (c4) .. (end) node[above,pos=0.37,yshift=2pt] {\sz $\ep$};
\begin{scope}[xshift=2.8*\u]
  \def\uu{2.5cm}
  \def\dd{2.5cm}
  \draw (0,0) -- (1*\uu,0);
  \foreach \i/\j in {0/{a_0},0.48/{a_1},0.65/{a'_1},1/{a_2}}
  {
  \node[below] at (\i*\uu,0) {\sz \rule{0pt}{7pt}$\j$};
  }
  \draw (0,0) -- (0,1*\dd);
  \draw (0,1*\d) -- ++(0.48*\uu,0) -- ++(0,-1*\dd);
  \draw[dashed] (0.6*\uu,0) -- ++(0,0.5*\dd);
  \draw (0.48*\uu,0.5*\dd) -- ++(0.52*\uu,0) -- ++(0,-0.5*\dd);
  \node[left] at (0,1*\dd) {\sz $b_0$};
  \node[right] at (0.48*\uu,0.9*\dd) {\sz $b'_0$};
  \node[left] at (0.48*\uu,0.5*\dd) {\sz $b_1$};
  \node[right] at (1*\uu,0.4*\dd) {\sz $b'_1$};
  \coordinate (sstart) at (0*\uu,1*\dd);
  \coordinate (p1) at (0.48*\uu,0.9*\dd);
  \coordinate (p2) at (0.6*\uu,0.5*\dd);
  \coordinate (eend) at (1*\uu,0.4*\dd);
  \coordinate (c1) at ($(sstart) + 0.05*(1*\uu,-1*\dd)$);
  \coordinate (c2) at ($(p1) - 0.05*(1*\uu,-1*\dd)$);
  \coordinate (c3) at ($(p1) + 0.05*(1*\uu,-1*\dd)$);
  \coordinate (c4) at ($(p2) - 0.05*(1*\uu,-1*\dd)$);
  \coordinate (c5) at ($(p2) + 0.05*(1*\uu,-1*\dd)$);
  \coordinate (c6) at ($(eend) - 0.05*(1*\uu,-1*\dd)$);
  \draw[red] (sstart)
    .. controls (c1) and (c2) .. (p1) node[black,below,pos=0.5] {\sz $u_0$}
    .. controls (c3) and (c4) .. (p2) node[black,right,pos=0.5] {\sz $u'_1$}
    .. controls (c5) and (c6) .. (eend) node[black,below,pos=0.5] {\sz $u_1$};
\end{scope}
\begin{scope}[xshift=4.5*\u]
  \def\uu{2.5cm}
  \def\dd{2.5cm}
  \draw (0,0) -- (1*\uu,0);
  \foreach \i/\j in {0/{a'_{n-1}},0.48/{a_n},0.65/{a'_n},1/{a_{n+1}}}
  {
  \node[below] at (\i*\uu,0) {\sz \rule{0pt}{7pt}$\j$};
  }
  \draw[dashed] (0,0) -- (0,1*\dd);
  \draw (0,1*\d) -- ++(0.48*\uu,0) -- ++(0,-1*\dd);
  \draw[dashed] (0.6*\uu,0) -- ++(0,0.5*\dd);
  \draw (0.48*\uu,0.5*\dd) -- ++(0.52*\uu,0) -- ++(0,-0.5*\dd);
  \node[left] at (0,1*\dd) {\sz $b_{n-1}$};
  \node[right] at (0.48*\uu,0.9*\dd) {\sz $b'_{n-1}$};
  \node[left] at (0.48*\uu,0.5*\dd) {\sz $b_{n}$};
  \node[right] at (1*\uu,0.4*\dd) {\sz $b'_{n}$};
  \coordinate (sstart) at (0*\uu,1*\dd);
  \coordinate (p1) at (0.48*\uu,0.9*\dd);
  \coordinate (p2) at (0.6*\uu,0.5*\dd);
  \coordinate (eend) at (1*\uu,0.4*\dd);
  \coordinate (c1) at ($(sstart) + 0.05*(1*\uu,-1*\dd)$);
  \coordinate (c2) at ($(p1) - 0.05*(1*\uu,-1*\dd)$);
  \coordinate (c3) at ($(p1) + 0.05*(1*\uu,-1*\dd)$);
  \coordinate (c4) at ($(p2) - 0.05*(1*\uu,-1*\dd)$);
  \coordinate (c5) at ($(p2) + 0.05*(1*\uu,-1*\dd)$);
  \coordinate (c6) at ($(eend) - 0.05*(1*\uu,-1*\dd)$);
  \draw[red] (sstart)
    .. controls (c1) and (c2) .. (p1) node[black,below,pos=0.5] {\sz $u_{n-1}$}
    .. controls (c3) and (c4) .. (p2) node[black,right,pos=0.5] {\sz $u'_{n}$}
    .. controls (c5) and (c6) .. (eend) node[black,below,pos=0.5] {\sz $u_{n}$};
\end{scope}
\end{tikzpicture}
\]
Let $b_0=\beta_0$. By assumption, $b_0\leq \ep(0)$. If $b_0<\ep(0)$, choose a positive number $a_1<1/2$ close enough to $a_0$ so that $b_0<\ep(a_1)$. If $b_0=\ep(0)=\ep_0$, choose any positive number $a_1<1/2$. Also choose a decreasing sequence of positive numbers $r_1>r_2>\dots$ with $r_n<\ep(a_{n+1})$. Then set $a_n=(n-1)/2$, $n=2,3,\dots$, and define $b_n=r_n/5^n$, $n=1,2,\dots$. We have
\begin{enumerate}[(i)]
\item
$b_0\leq \ep(a_1)$ with $b_0 < \ep(a_1)$ unless $\ep(x)=\ep_0$ is constant, and $b_n<\ep(a_{n+1})$ for $n=1,2,\dots$

\item
$\sum_{i=n+1}^\infty b_i < b_n/4$, $n=0,1,\dots$
\end{enumerate}
In (i), the bounds on $b_0$ hold by the choice of $a_1$ and $b_n<\ep(a_{n+1})$ is clear from the definition of $b_n$ for $n>0$. For (ii), we have
$\sum_{i=n+1}^\infty b_i = \sum_{i=n+1}^\infty r_i/5^i < \sum_{i=n+1}^\infty r_{n+1}/5^i
= r_{n+1}/(4\cdot 5^n) < r_n/(4\cdot 5^n) = b_n/4$.

Adding $b_n$ to both sides of (ii), we get $\sum_{i=n}^\infty b_i < (5/4)b_n < (4/3)b_n$, so $\sum_{i=n}^\infty (3/4)b_i < b_n$.
Choose $a'_n$, $n=1,2,\dots$, and $b'_n$, $n=0,1,\dots$, so that
$a_n<a'_n<a_{n+1}$ and $b_{n+1}<b'_n<b_n$, and $a'_n$, $b'_n$ are close enough to $a_n$, $b_n$, respectively, so that
\begin{enumerate}[(i)]
\setcounter{enumi}{2}
\item
$a'_n-a_n<1/4$, $n=1,2,\dots$

\item
$b_0\leq\ep(a'_1)$ with $b_0 <\ep(a'_1)$ unless $\ep(x)$ is constant, and $b_n<\ep(a'_{n+1})$, $n=1,2,\dots$

\item
$b_{n-1}'(a'_{n}-a_{n})+\sum_{i=n}^\infty (3/4)b_i < b'_n$, $n=1,2,\dots$
\end{enumerate}
Apply Proposition \ref{p:approx.n.der.incr} with $n=0$ and $s=-1$ (the functions $f$ to which we apply the proposition can be taken to be the affine functions  having the correct values at the endpoints) to get $C^\infty$ functions $u_{0}\colon[a_{0},a_{1}]\to\R$, $u'_{n}\colon[a_{n},a'_{n}]\to\R$, $n\geq 1$,  $u_{n}\colon[a'_{n},a_{n+1}]\to\R$, $n\geq 1$, so that $D^ju_0(a_0)=\beta_j$, $j=0,1,\dots$, $u_0(a_1)=b'_0$, $u'_n(a_n)=b'_{n-1}$, $u'_n(a'_n)=b_n$, $u_n(a'_n)=b_{n}$, $u_n(a_{n+1})=b'_{n}$, and at each endpoint $c\not=a_0$ of the domain of a function $g$ of the form $u_0$, $u_n$, or $u'_n$,  $Dg(c)=-1$ and $D^jg(c)=0$, $j>1$, and when $x$ is not an endpoint, $Dg(x)<0$. The function $u=u_0\cup\bigcup_{n=1}^\infty (u_n\cup u'_n)$ satisfies (1) and (2), so there remains to verify (3).

We first verify the second inequality in (3). For $n\geq 2$ and $a'_{n-1}\leq x\leq a'_{n}$ we have $u(x)\leq b_{n-1}<\ep(a'_n)\leq \ep(x)$.
Similarly for $a_0\leq x\leq a'_1$ we get $u(x)<\ep(x)$: if $\ep(x)$ is not constant, then we have $u(x)\leq b_0<\ep(a'_1)\leq \ep(x)$, and if $\ep(x)=\ep_0$ is constant, then for $a_0<x\leq a_1'$ we have $u(x)<b_0\leq \ep_0 = \ep(x)$.

For the first inequality in (3), for $a'_n\leq x\leq a_{n+1}$, $n\geq 1$, we have
\begin{align*}
\int_x^\infty u(t)\,dt & \leq \sum_{i=n}^\infty \int_{a'_i}^{a'_{i+1}} u(t)\,dt \leq \sum_{i=n}^\infty \int_{a'_i}^{a'_{i+1}} b_i\,dt \\
& \leq \sum_{i=n}^\infty (3/4) b_i < b'_n \leq u(x)
\end{align*}
Similarly for $a_0\leq x\leq a_1$, letting the first integral in the first sum be $\int_{a_0}^{a'_{1}} u(t)\,dt$ (which is $<b_0(a'_1-a_0) < (3/4)b_0$), we get $\int_x^\infty u(t)\,dt<u(x)$.
For $a_n\leq x\leq a'_{n}$, $n\geq 1$, we have
\begin{align*}
\int_x^\infty u(t)\,dt & \leq \int_{a_n}^{a'_{n}} u(t)\,dt + \sum_{i=n}^\infty \int_{a'_i}^{a'_{i+1}} u(t)\,dt \\
& \leq b'_{n-1}(a'_n-a_n)+ \sum_{i=n}^\infty (3/4) b_i < b_n \leq u(x)\qedhere
\end{align*}
\end{proof}

\begin{rem}\label{r::u.drops.quickly.to.zero}
Write $T$ for the transformation on positive continuous functions $u\colon[0,\infty)\to \R$, together with the constant function on $[0,\infty)$ with value $\infty$, given by
\[
T(u)(x)=\int_x^\infty u(t)\,dt.
\]
Note that $T$ is monotone in the sense that if $u<v$ (i.e., $u(x)<v(x)$ for all $x$) and $\int_0^\infty v(t)\,dt<\infty$, then $T(u)<T(v)$. (3) gives $T(u)<u$ (and so $\int_0^\infty u(t)\,dt = T(u)(0)<u(0)<\infty$), and then by repeatedly applying $T$ to $T(u)<u$ we get that for all $k=0,1,\dots$,
\[
T^k(u)<T^{k-1}(u)<\dots<T^2(u)<T(u)<u<\ep.
\]
\end{rem}

\begin{lem}\label{l:approx.a.poly.on.platform}
Let $n$ be a nonnegative integer. Let $r,s\in\{-1,1\}$. Let $I=[0,\infty)$.
Let $\ep(x)$ be a positive continuous function on $rI$. Fix a positive number $b \leq \ep(0)$, with $b<\ep(0)$ unless $\ep(x)=\ep_0$ is constant. Let $\beta_j$, $j=0,1,\dots$ be a sequence of real numbers such that $\beta_0=-sr b$ and either $\beta_j=0$ for all $j>0$, or for the least $j>0$ with $\beta_j\not=0$ we have $sr^{j+1}\beta_j > 0$.
Then there is a $C^\infty$ function $g\colon rI\to\R$ such that
\begin{enumerate}
\item
$s D^{n+1}g(x)>0$, $x\in rI$, $x\not=0$.

\item
$D^{n+j}g(0)=\beta_j$, $j=0,1,\dots$

\item
$|D^jg(x)|<\ep(x)$, $x\in rI$, $x\not=0$, $j=0,\dots,n$.
\end{enumerate}
Now assume that $\lim_{x\to r\infty}\ep(x)=0$. From {\rm(1)} and {\rm(3)} we get {\rm(4)} and {\rm(5)} for $j=0,\dots,n$.
\begin{enumerate}
\setcounter{enumi}{3}
\item
$s (-r)^{n+j+1}D^{j}g(x) > 0$, for all $x\in rI$.

\item
If $1\leq j\leq n$, we have $D^{j-1}g(0) = -r\int_{rI} D^{j}g$.
\end{enumerate}
\end{lem}

\begin{proof}
We can assume $r=s=1$. (Given that case, define $\bar{\ep}(x)=\ep(rx)$, $\bar{\beta}_j = sr^{j+1}\beta_j$, so that $\bar{\beta}_0 = sr \beta_0 = sr (-sr b) = -b$. Applying the case $r=s=1$ produces a function $\bar{g}$ on $I$ satisfying (1)--(3) (with $r=s=1$), and also (4), (5) when $\lim_{x\to r\infty}\ep(x)=0$. Then $g$ on $rI$ given by $g(x)=sr^n\bar{g}(rx)$ is as desired.)

To define $g$ satisfying (1), (2), (3), first use Lemma \ref{l:u.drops.quickly.to.zero} with $\ep(x)=b$ to get a positive $C^\infty$ function $u\colon[0,\infty)\to\R$ so that $Du(x)<0$, $x>0$, $u(0)=b$, $D^{j}u(0)=-\beta_j$, $j=1,2,\dots$, and
$T^i(u)(x) < \ep(x)$, $i=0,1,\dots,n$, where $T$ is the operator from Remark \ref{r::u.drops.quickly.to.zero}.
Then let
\[
g = (-1)^{n+1}T^{n}(u).
\]
Since for $T(u)(x)=\int_x^\infty u$ we have $D(T(u))=-u$, we get by induction on $j=0,\dots,n$ that
\[
D^jg = (-1)^{n+j+1}T^{n-j}(u).
\]
We have (3) by the choice of $u$, and (1) and (2) follow easily from the fact that $D^ng = -u$.

We now verify that (4) and (5) follow from (1) and (3) when $\lim_{x\to \infty}\ep(x)=0$. For (4), we verify by induction that for $j=0,\dots,n$,
\[
(-1)^{j+1}D^{n-j}g(x) > 0,\ x\geq 0.
\]
When $j=0$, this says that $D^ng(x) < 0$.
By (1), $D^ng$ is strictly increasing, so if there were a point $x_0\in[0,\infty)$ where $D^ng(x_0)\geq 0$, then at any point $x_1>x_0$, we have $D^ng(x_1) > 0$. Then for $x\geq x_1$ we have $\ep(x)>|D^ng(x)| = D^ng(x)\geq D^ng(x_1)$, so that $\lim_{x\to\infty}\ep(x)=0$ fails.

For the induction step, if $j<n$ and $(-1)^{j+1}D^{n-j}g(x) > 0$ when $x\geq 0$, then we want to show that
\[
(-1)^{j}D^{n-j-1}g(x) > 0,\ \ x\geq 0.
\]
Suppose this fails at a point $x_0$. By the induction hypothesis, the function
\begin{align*}
d(x)&=(-1)^{j}D^{n-j-1}g(x) \\
&=-(-1)^{j+1}D^{n-j-1}g(x)
\end{align*}
has a negative derivative and hence is strictly decreasing. From $d(x_0)\leq 0$, we get $d(x_1)<0$ for any $x_1>x_0$. Then for $x\geq x_1$ we have
\[
\ep(x) > |D^{n-j-1}g(x)| = |(-1)^{j}D^{n-j-1}g(x)| = -d(x)\geq -d(x_1)
\]
and hence $\lim_{x\to\infty}\ep(x)=0$ fails.
Next, we verify (5). For $j<n$ and $x\in I$, we have
\begin{align*}
\ep(x) & \geq |D^{n-j-1}g(x)| = (-1)^{j}D^{n-j-1}g(x) \\
& = (-1)^{j}D^{n-j-1}g(0) - \int_0^x (-1)^{j+1}D^{n-j}g(t)\,dt
\end{align*}
Since the right-hand sides are nonnegative, the (nonnegative) integrals
\[
\int_0^x (-1)^{j+1}D^{n-j}g(t)\,dt
\]
are bounded above by $(-1)^{j}D^{n-j-1}g(0)$, and taking $\lim_{x\to\infty}$ gives
\[
(-1)^{j}D^{n-j-1}g(0) = \int_0^\infty (-1)^{j+1}D^{n-j}g(t)\,dt
\]
which simplifies to
\[
D^{n-j-1}g(0) = - \int_0^\infty D^{n-j}g(t)\,dt.
\]
Thus, (5) holds.
\end{proof}

\begin{thm}\label{t:approx.a.poly.on.platform}
Let $f$ be a polynomial with $D^nf$ having constant value $k$, for some nonnegative integer $n$. Let $r,s\in\{-1,1\}$. Let $a\in\R$, $I=[a,\infty)$.
Let $\ep(x)$ be a positive continuous function on $rI$. Fix a positive number $b \leq \ep(ra)$, with $b<\ep(ra)$ unless $\ep(x)=\ep_0$ is constant. Let $\beta_j$, $j=0,1,\dots$ be a sequence of real numbers such that $\beta_0=k-sr b$ and either $\beta_j=0$ for all $j>0$, or for the least $j>0$ with $\beta_j\not=0$ we have $sr^{j+1}\beta_j > 0$.
Then there is a $C^\infty$ function $g\colon rI\to\R$ such that
\begin{enumerate}
\item\label{3.4.1}
$s D^{n+1}g(x)>0$, $x\in rI$, $x\not=ra$.

\item\label{3.4.2}
$D^{n+j}g(ra)=\beta_j$, $j=0,1,\dots$

\item\label{3.4.3}
$|D^jg(x)-D^jf(x)|<\ep(x)$, $x\in rI$, $x\not=ra$, $j=0,\dots,n$.
\end{enumerate}
Now assume that $\lim_{x\to r\infty}\ep(x)=0$. From {\rm(1)} and {\rm(3)} we get {\rm(4)} and {\rm(5)} for $j=0,\dots,n$.
\begin{enumerate}
\setcounter{enumi}{3}
\item\label{3.4.4}
$s (-r)^{n+j+1}D^{j}g(x) > s(-r)^{n+j+1}D^{j}f(x)$, for all $x\in rI$.

\item\label{3.4.5}
If $1\leq j\leq n$, we have $D^{j-1}g(ra)-D^{j-1}f(ra) = -r\int_{rI} (D^{j}g-D^{j}f)$.
\end{enumerate}
If $D^jf$, $j=0,\dots,n$, are all monotone on $rI$, and those which are not constant are nonzero at $ra$, then we can ask moreover that
\begin{enumerate}
\setcounter{enumi}{5}
\item\label{3.4.6}
$D^jg(x)\not=0$, $x\in rI$, $j=1,\dots,n$.
\end{enumerate}
\end{thm}

\begin{proof}
With the given values of $r,s,b$ and defining $\ep_1(x)=\ep(x-ra)$ on $r[0,\infty)$, get a $C^\infty$ function $g_1\colon r[0,\infty)\to\R$ from Lemma \ref{l:approx.a.poly.on.platform}. Then $g\colon rI\to\R$ given by $g(x) = f(x) + g_1(x-ra)$ is $C^\infty$ since $f$ is a polynomial. We have $D^ng(x) = D^nf(x) + D^ng_1(x-ra) = k + D^ng_1(x-ra)$ and hence $D^ng(ra) = k + D^ng_1(0) = k - srb = \beta_0$. When $j>0$, we have $D^{n+j}g(x) = D^{n+j}g_1(x-ra)$, so (2) follows from clause (2) for $g_1$ in Lemma \ref{l:approx.a.poly.on.platform}. Clause (3) follows immediately from the corresponding property for $g_1$ in Lemma \ref{l:approx.a.poly.on.platform}, as do (4) and (5) by rephrasing them as statements about the difference $g(x) - f(x) = g_1(x-ra)$.

There remains to show that (6) can be arranged when $D^jf$, $j=0,\dots,n$, are all monotone on $rI$, and those which are not constant are nonzero at $ra$.
If $f=0$, (4) gives immediately $D^jg(x)\not=0$, $x\in rI$, $j=0,\dots,n$, for any $g$ obtained from the theorem. Assume now that $f$ is nonzero with degree $m\leq n$. The derivatives which are not constant are strictly monotone on $rI$, and for $j=1,\dots,m$, $D^jf$ has no zeros on $rI\sm\{ra\}$ (Proposition \ref{p:monotone.der.sign}). By assumption, $D^jf(ra)\not=0$ for all $j=1,\dots,m$.
Choose $\ep(x)<\min\{|D^jf(x)|:j=1,\dots,m\}$, $x\in rI$, and get $g$ from the theorem. From (3) we see that for $j=1,\dots,m$, $D^jg(x)$ has the same sign as $D^jf(x)$ for all $x\in rI$, in particular, $D^jg(x)\not=0$. For $m<j\leq n$, we have $D^jf(x)=0$ and hence from (4) we get that $D^jg(x)\not=0$, $x\in rI$.
\end{proof}

The following remark shows the need for assuming in (6) that nonconstant derivatives are nonzero at $ra$ but also shows that the assumption can be weakened to say that we need this assumption for every second derivative.

\begin{rem}\label{r:0}
(a) For $a=0$ and $n=3$, and with $r=s=1$, consider $f(x)=1+x^3$ on $[0,\infty)$, with any continuous positive $\ep(x)$ so that $\lim_{x\to\infty}\ep(x)=0$. The functions $D^jf$, $j=0,1,2,3$, are all monotone on $[0,\infty)$. The values $f(0)$ and $D^3f(0)$ are nonzero, however, $Df(0)=D^2f(0)=0$. An approximation $g$ as in the theorem above (taking any $b$ such that $0<b<\ep(0)$) will have $Dg(0)<0$ by (4) but $Dg(x)$ is eventually positive by (3), so $Dg$ has a zero on $(0,\infty)$.

(b) Suppose that $f$ has degree $m\leq n$ and that the derivatives $D^jf$, $j=0,\dots,m$, are all monotone on $rI$. The derivatives of order $<m$ are strictly monotone, and those of order $1,\dots,m$ have no zero on $rI$ except possibly at $ra$.
If $t\in\{1,-1\}$ is the sign of the coefficient of the term of highest degree in $f$, then the derivatives $D^jf$, $j=1,\dots,m$, have constant sign $r^{m-j}t$ on $rI\sm\{ra\}$ (Proposition \ref{p:monotone.der.sign}).
From (4) we get $s (-r)^{n-j}D^{j}g(x)<s (-r)^{n-j}D^{j}f(x)$, for all $x\in rI$. The sign of the right hand side, when $x\not=ra$, is $s (-r)^{n-j} r^{m-j}t$, which alternates between $1$ and $-1$. For those $j=1,\dots,m$ for which it equals $-1$, we have $s (-r)^{n-j}D^{j}g(x)<s (-r)^{n-j}D^{j}f(x)\leq 0$ and hence $D^jg(x)\not=0$, regardless of whether $D^jf(ra)$ is zero. Letting
\[
S=\{j\in\{1,\dots,m\}:s (-r)^{n-j}r^{m-j}t = 1\},
\]
we have that as long as $D^jf(ra)\not=0$ when $j\in S$, we can find $g$ satisfying (6). The proof is as for (6), but taking $\ep(x)<\min\{|D^jf(x)|:j\in S\}$.

(c) In the example in part (a),
$s (-r)^{n-j}r^{m-j}t = (-1)^{3-j}$, which for $j=1,2,3$ equals $-1$ when $j=2$. Hence, in that example, we would have $D^2g(x)\not=0$ on $[0,\infty)$ even though $D^2f(0)=0$.

(d) Taking $a=0$ and $n=2$, with $r=s=1$, consider $f(x)=1+x^2$ on $[0,\infty)$, with any continuous positive $\ep(x)$ so that $\lim_{x\to\infty}\ep(x)=0$. The functions $D^jf$, $j=0,1,2$, are all monotone on $[0,\infty)$. The values $f(0)$ and $D^2f(0)$ are nonzero, and $Df(0)=0$.
We have $s (-r)^{n-j}r^{m-j}t = (-1)^{2-j}$, which equals $-1$ when $j=1$. Hence, we get $Dg(x)\not=0$ on $[0,\infty)$ even though $Df(0)=0$.
\end{rem}

\begin{thm}\label{t:approx.a.poly.on.platform.2}
Let $f$ be a polynomial with $D^nf$ having constant value $k$, for some nonnegative integer $n$. Let $r,s\in\{-1,1\}$. Let $a\in\R$, $I=[a,\infty)$.
Let $\ep>0$. Let $\beta_j$, $j=1,2,\dots$ be a sequence of real numbers such either $\beta_j=0$ for all $j=1,2,\dots$, or for the least $j$ with $\beta_j\not=0$ we have $sr^{j+1}\beta_j > 0$.
Then there is a $C^\infty$ function $h\colon rI\to\R$ such that
\begin{enumerate}
\item\label{c.1}
$s D^{n+1}h(x)>0$, $x\in rI$, $x\not=ra$.

\item\label{c.2}
$D^{n+j}h(ra)=\beta_j$, $j=1,2,\dots$

\item\label{c.3}
$D^jh(ra) = D^jf(ra)$, $j=0,\dots,n$.

\item\label{c.4}
$|D^nh(x)-D^nf(x)| = |D^nh(x)-k|<\ep$, $x\in rI$.
\end{enumerate}
If $D^jf$ is monotone on $rI$ for each $j=0,\dots,n$ then there is a choice of $s\in\{1,-1\}$ such that for any $C^\infty$ function $h\colon rI\to\R$ satisfying \ref{c.1} and \ref{c.3} also satisfies for each $j=0,\dots,n$,
\begin{enumerate}
\setcounter{enumi}{4}
\item\label{c.5}
$D^{j+1}h(x)\not=0$ for all $x\in rI\sm\{ra\}$.

\item\label{c.6}
$D^jh$ has the same monotonicity as $D^jf$ on $rI$.
\end{enumerate}
\end{thm}

\begin{proof}
Apply Theorem \ref{t:approx.a.poly.on.platform} with $\ep(x)$ replaced by the constant $\ep/2$ and taking $b$ so that $0<b<\ep/2$. This yields a $C^\infty$ function $g\colon rI\to\R$ such that
\begin{enumerate}[(a)]
\item
$s D^{n+1}g(x)>0$, $x\in rI$, $x\not=ra$.

\item
$D^{n+j}g(ra)=\beta_j$, $j=0,1,2,\dots$ (where $\beta_0=k-srb$).

\item
$|D^ng(x)-D^nf(x)| = |D^ng(x)-k|<\ep/2$, $x\in rI$, $x\not=ra$.
\end{enumerate}
Define $h(x) = g(x) + \sum_{j=0}^n (D^j(f-g)(ra)/j!)(x-ra)^j$.
For $j=0,\dots,n$ we have
\[
D^jh(ra) = D^jg(ra) + D^j(f-g)(ra) = D^jf(ra).
\]
so that \ref{c.3} holds, and for $j > n$, $D^jh(x) = D^jg(x)$, so that (a) and (b) yield \ref{c.1} and \ref{c.2}. For \ref{c.4}, we have $D^nh(ra) = D^nf(ra)$, so \ref{c.4} holds for $x=ra$, and if $x\in rI$ is not equal to $ra$, then using (c) we have
\begin{align*}
|D^nh(x) - D^nf(x)| & = |D^ng(x) + D^n(f-g)(ra) - k| \\
& = |D^ng(x) - D^ng(ra)| = |D^ng(x) - \beta_0| \\
& = |D^ng(x) - k + srb| \leq |D^ng(x) - k| + b < \ep
\end{align*}

For the last part of the theorem, suppose that each $D^jf$, $j=0,\dots,n$, is monotone on $rI$ and let $h\colon rI\to \R$ be a $C^\infty$ function satisfying \ref{c.1} and \ref{c.3}.

When $f=0$, we only need to verify that for each $j=0,\dots,n$, $D^{j+1}h$ is nonzero on $rI\sm\{ra\}$ since that implies that $D^{j}h$ is monotone on $rI$. In this case $s\in\{1,-1\}$ can be chosen arbitrarily and we verify that
for $j=1,\dots,n+1$, $D^jh$ has sign $r^{n-j+1}s$ on $rI\sm\{ra\}$. This holds for $j=n+1$ by \ref{c.1}, and if $0\leq j\leq n$ and it holds for $j+1$, then since $D^jh(ra)=0$ by \ref{c.3}, for $x\in rI\sm\{ra\}$ we have
\[
D^jh(rx) = \int_{ra}^{rx} D^{j+1}h(t)\,dt = r\int_{r[a,x)} D^{j+1}h
\]
which has sign $r(r^{n-j}s) = r^{n-j+1}s$. Thus, for $j=0,\dots n$, $D^{j+1}h$ is nonzero on $rI\sm\{ra\}$.

Now assume that $f$ is nonzero of degree $m\leq n$ with its coefficient of $x^m$ having sign $t$.
Let $s=r^{n-m+1}t$.

\begin{claim}\label{claim:Djh.sign}
For $j=1,\dots,n+1$, $D^jh(x)$ has sign $r^{j-m}t$ $(=r^{m-j}t)$ on $rI\sm\{ra\}$.
\end{claim}

\begin{proof}
When $j=n+1$, this is true by \ref{c.1}. If $m < j\leq n$ and the claim holds for $j+1$, then proceed as in the case $f=0$ to show that the claim holds for $j$ as well.  
If $1\leq j\leq m$ and the claim holds for $j+1$, then for $x\in I$, $x>a$, we have
\[
D^jh(rx) = D^jh(ra) + \int_{ra}^{rx} D^{j+1}h(t)\,dt = D^jf(ra) + r\int_{r[a,x)} D^{j+1}h.
\]
The term $D^jf(ra)$ is either zero or of sign $r^{m-j}t$ (Proposition \ref{p:monotone.der.sign}), and the second term is of sign $r(r^{m-j-1}t) = r^{m-j}t$, so $D^jh(rx)$ has sign $r^{m-j}t$, as desired.
\end{proof}

Property \ref{c.5} follows directly from Claim \ref{claim:Djh.sign}. For \ref{c.6}, given $j=0,\dots,n$, if $j\geq m$ then $D^jf$ is constant on $rI$ and $D^jh$ is monotone by \ref{c.5}, so \ref{c.6} follows.
If $0\leq j<m$, then $D^{j+1}f$ and $D^{j+1}h$ have the same sign on $rI\sm\{ra\}$ by Proposition \ref{p:monotone.der.sign} and \ref{claim:Djh.sign}, and again \ref{c.6} follows.
\end{proof}

\section{Simultaneous comonotone approximation of a $C^n$ function and its derivatives}
\label{s:Simultaneous comonotone approximation of a $C^n$ function and its derivatives}

In this section, we prove Theorem \ref{t:C.n.Hoischen.C.n.fcts} which was stated in the introduction. This theorem makes use of a common witnessing set for the piecewise monotonicity of $D^jf$, $j=0,\dots,n$, which has no more than one point on any platform of $D^nf$, and hence (by Corollary \ref{c:platforms.of.derivatives}) has no more than one point on any platform of $D^if$ for $i\leq n$.
The following proposition shows (and it is easy to see that) there is such a set when the platforms of $D^nf$ are singletons, or equivalently, there are no intervals on which $f$ is the zero function or a polynomial of degree $\leq n$. The proposition also shows that there is such a witnessing set when $n\leq 2$.
Note that some assumption is needed because if on some interval $I$, $f$ is a polynomial of degree $n\geq 3$, then $f$ could have more than one turning point on $I$ which is contained in a platform of $D^nf$.

\begin{prop}\label{p:common.witness.1}
Let $I$ be a nontrivial interval of $\R$ and let $f\colon I\to\R$ be a $C^n$ function for which $D^nf$ is piecewise monotone.
\begin{enumerate}
\item
If $D^nf$ has singleton platforms, then there is a set $E$ which is a witness to the piecewise monotonicity of each of $D^jf$, $j=0,\dots,n$.

\item
If $n\leq 2$, then there is a set $E$ which is a witness to the piecewise monotonicity of $D^jf$ for $j\leq n$ and has at most one point on any platform for $D^nf$.
\end{enumerate}
\end{prop}

    Recall from Proposition \ref{p:f.not.piecewise.monotone.implies.Df.also} that piecewise monotonicity of $D^nf$ implies the piecewise monotonicity of $D^if$ for all $i\leq n$.

\begin{proof}
(1) By Corollary \ref{c:platforms.of.derivatives}, the platforms of all $D^jf$, $j=0,\dots,n$, are singletons. The set $E$ of all turning points of all these derivatives has the desired property by Proposition \ref{witnessing set}.

(2) For $n=0$ this is trivial. Suppose $n=1$.
Let $E_0$ be a set consisting of one point chosen from each turning platform of $f$. Now get a set $E\sq I$ from $E_0$ by adding to $E_0$ one point from each turning platform for $Df$ which is disjoint from $E_0$.
We must check that if $P$ is a platform for $Df$ of positive length, then $E\cap P$ has at most one point. On the interval $P$, $f$ has a constant derivative, so $f(x)=ax+b$ for some $a,b\in\R$. If $a\not=0$, then no points of $P$ are turning points for $f$, so $E_0\cap P=\e$ and hence $E\cap P$ has at most one point by definition of $E$. If $a=0$, then $f$ is constant on $P$, so $P$ is contained in a platform of $f$. By Corollary \ref{c:platforms.of.derivatives}, $P$ is in fact a platform of $f$. Either $E_0\cap P$ has exactly one point, in which case $E\cap P = E_0\cap P$, or $E_0\cap P$ is empty, in which case $E\cap P$ has at most one point.

Now suppose $n=2$. As in the case $n=1$, let $E_0$ be a set consisting of one point chosen from each turning platform of $f$. Let $E_1\sq I$ be obtained from $E_0$ by adding to $E_0$ one point from each turning platform for $Df$ which is disjoint from $E_0$.
Then let $E\sq I$ be obtained from $E_1$ by adding to $E_1$ one point from each turning platform for $D^2f$ which is disjoint from $E_1$.
We must check that if $P$ is a platform of positive length for $D^2f$, then $E\cap P$ has at most one point. On the interval $P$, $D^2f$ is constant, so  $f(x)=ax^2+bx+c$ for some $a,b,c\in \R$. If $a\not=0$, then there is at most one turning point for $f$ in $P$, and no turning points for $Df(x)=2ax+b$, so $E_1\cap P=E_0\cap P$. If this set is a singleton, then by definition of $E$, it equals $E\cap P$. If this set is empty, then again by definition of $E$, $E\cap P$ is at most a singleton. If $a=0$, $b\not=0$, then $Df(x)=b\not=0$ on $P$, so there are no turning points for $f$ in $P$ and all points of $P$ are on the same platform of $Df$, so $E_1\cap P$ has at most one point and therefore $E\cap P$ has at most one point.
If $a=b=0$, then $f$ is constant on $P$, so $P$ is a platform of all of $f$, $Df$, $D^2f$ by Corollary \ref{c:platforms.of.derivatives}. If $E_0\cap P$ is not empty, then $E_1\cap P=E_0\cap P$ and then $E\cap P=E_1\cap P$. If $E_0\cap P=\e$, then either $E_1\cap P=\e$ in which case $E\cap P$ has at most one point, or $E_1\cap P$ is a singleton, in which case $E\cap P=E_1\cap P$.
\end{proof}

Recall from the introduction that in the statement of Theorem \ref{t:C.n.Hoischen.C.n.fcts}, the sets $P^{D^nf}_{\min}$ and $P^{D^nf}_{\max}$ are denoted $P^{n}_{\min}$ and $P^{n}_{\max}$, respectively.
We also set $e_{\min} = \inf E$, $e_{\max} = \sup E$, and $W_n = \{x\in\R:
e_{\max}\in P^n_{\max}$ and $x>e_{\max}$, or $e_{\min}\in P^n_{\min}$ and $x<e_{\min}\}$.

\begin{thm}
\label{t:C.n.Hoischen.C.n.fcts}
Assume $(Q_n)$. Let $n$ be a nonnegative integer.  Let $f\colon \R\to\R$ be a $C^n$ function such that $D^nf$ is piecewise monotone and nonconstant. Let $E\sq \R$ be a closed discrete set which has no more than one point on any platform of $D^nf$, and has a point on each turning platform of $D^jf$, $j=0,\dots,n$. Let $\ep\colon \R\to\R$ be a positive continuous function and let $\ep_0>0$.
Then there is a function $g\colon \R\to\R$ which is the restriction of an entire function and satisfies the following conditions for $x\in \R$.
\begin{enumerate}
\item\label{p:C.n.Hoischen:5}
$D^jg(x)=D^jf(x)$ when $x\in E$, $j=0,\dots,n$.

\item\label{p:C.n.Hoischen:1a}
$D^jg(x)\not=0$ when $x\not\in E$, $1\leq j\leq n+1$.

\item\label{p:C.n.Hoischen:2}
$D^{n+1}g(x)\not=0$ when $x\in E\sm(P^n_{\min}\cup P^n_{\max})$ and $x$ not a turning point of $D^nf$

\item\label{p:C.n.Hoischen:3}
$|D^ng(x)-D^nf(x)|<\ep_0$.

\item\label{p:C.n.Hoischen:4}
$|D^jg(x)-D^jf(x)|<\ep(x)$, when $x\notin W_n$, $j=0,\dots,n$.

\item\label{p:C.n.Hoischen:6}
$D^jg$ is comonotone with $D^jf$, with $E$ as witnessing set, $j=0,\dots,n$.
\end{enumerate}
\end{thm}

\begin{proof}
By Corollary \ref{c:platforms.of.derivatives}, each platform of $D^jf$, $j=0,\dots,n$, is contained in a platform of $D^nf$ so that if a set has at most one point on each platform of $D^nf$, then it also has at most one point on each platform of each $D^jf$, $j=0,\dots,n$.

We may assume that $\ep(x)\leq \ep_0$ for all $x\in\R$ and $\lim_{x\to-\infty} \ep(x)=0$ and $\lim_{x\to\infty} \ep(x)=0$.
Since $D^nf$ is not constant, it takes on continuum many values and hence has continuum many distinct platforms where it has nonzero values.
Choose three such platforms different from $P^n_{\max}$, $P^n_{\min}$, and different from the countably many platforms in the zero sets $Z_{D^jf}$ of $D^jf$, $0\leq j\leq n$, (see Proposition \ref{p:piecewise.monotone.characterization} (2)), as well as the countably many platforms of $D^nf$ for the points of $E$.
By adding to $E$ one point from each of these three platforms,
we may assume that $E$ contains at least three points since the theorem for the original $E$ follows from the theorem for this augmented $E$.
(For one of these new points $x$, the requirement $D^jg(x)\not=0$ in \ref{p:C.n.Hoischen:1a} or \ref{p:C.n.Hoischen:2} is ensured by \ref{p:C.n.Hoischen:5} for the augmented set $E$ since $x\notin Z_{D^jf}$.)

For $a\in E$, let $a^-$ denote the immediate predecessor of $a$ in $E\cup\{-\infty\}$, and let $a^+$ denote the immediate successor of $a$ in $E\cup\{\infty\}$.

When $e_{\max}\in E$, on $[e_{\max},\infty)$, $D^nf$ is monotone. If $P^n_{\max}=\e$ then, using Proposition \ref{p:non.zero.point}, we can add to the set $E$ above $e_{\max}$ an unbounded increasing sequence of points $\{a_i\}_{i=1}^\infty$ with $D^jf(a_i)\not=0$, $j=0,\dots,n$, $i=1,2,\dots$, so that no two of the points $a_0=e_{\max}<a_1<a_2<\dots$ are on the same platform of $D^nf$. Then $e_{\max}=\infty$ for this augmented $E$, and \ref{p:C.n.Hoischen:5} for this augmented $E$ ensures \ref{p:C.n.Hoischen:1a} and \ref{p:C.n.Hoischen:2} for the original $E$.
When $P^n_{\rm  max}\not=\e$, then $e_{\max}\in E$ since $E$ is a nonvoid closed discrete set with at most one point on $P^n_{\rm  max}$. Thus, we may assume that $e_{\max}\in E$ if and only if $P^n_{\max}\not=\e$. Similarly, we can assume that $e_{\min}\in E$ if and only if $P^n_{\min}\not=\e$.

If $P^n_{\max}\not=\e$, let $s_{\max}$ be the value of $s$ as in the second part of Theorem \ref{t:approx.a.poly.on.platform.2} applied with $r=1$, $I=rI=[e_{\max},\infty)$. (The values of $\ep$ and the $\beta_j$ are not relevant to the second part of the theorem.)
Similarly, if $P^n_{\min}\not=\e$, let $s_{\min}$ be the corresponding value of $s$, this time applying Theorem \ref{t:approx.a.poly.on.platform.2} with $r=-1$, $rI=(-\infty,e_{\min}]$.

Let $\I$ be the collection of all closed intervals $I$ of the partition determined by $E$. So $\I=\{[a,a^+]:a\in E,\,a\not=e_{\max}\}\cup \{(-\infty,a]:a=e_{\min}$ and $e_{\min}\in E\}\cup
\{[a,\infty):a=e_{\max}$ and $e_{\max}\in E\}$. For each $a\in E$. write $I_L(a)$, $I_R(a)$ for the elements of $\I$ so that $\max I_L(a) = a = \min I_R(a)$.

For each of the (at most two) unbounded intervals $I\in\I$ with endpoint $a\notin P^n_{\min}\cup P^n_{\max}$, and each $j=0,\dots,n$, we have again that $D^jf$ is monotone but not constant on $I$. Fix a point $b_{I,j}\in I$ so that $D^jf(a)\not=D^jf(b_{I,j})$. We can take $\ep(x)$ so that the values $\ep(b_{I,j})$ are small enough so that for each $I,j$,
\begin{equation}\label{eq:bIj}
|y-D^jf(b_{I,j})|<\ep(b_{I,j})\ \text{implies that $y-D^jf(a)$ has the same sign as}\ D^jf(b_{I,j})-D^jf(a).
\end{equation}

\begin{claim}\label{claim:g_I.properties}
There are $C^\infty$ functions $g_I\colon I\to\R$, $I\in\I$, so that for $s\in\{1,-1\}$, $k\in\{L,R\}$, and $a\in E$ we have the following, where $I_L=I_L(a)$, $I_R=I_R(a)$.
\begin{enumerate}[label=(\roman*)]
\item\label{cl.gI.1}
For $I\in\I$, $D^jg_I(x)\not=0$, $j=1,\dots,n+1$, if $x$ is not an endpoint of $I$.

\item\label{cl.gI.1a}
$D^jg_{I_k}(a)=D^jf(a)$, $j=0,\dots,n$.

\item\label{cl.gI.2}
If $a\notin P^n_{\min}\cup P^n_{\max}$ and $sD^nf$ is increasing on $I_L\cup I_R$, then $D^{n+1}g_{I_k}(a)=s$ and $D^jg_{I_k}(a)=0$, $j > n+1$.

\item\label{cl.gI.2a}
If $a = e_{\min}\in P^n_{\min}$, $s_{\min}=s$ and $sD^nf$ is increasing on $I_R$, or $a = e_{\max}\in P^n_{\max}$, $s_{\max}=s$ and $sD^nf$ is increasing on $I_L$, then $D^{n+1}g_{I_k}(a)=s$ and $D^jg_{I_k}(a)=0$, $j > n+1$.

\item\label{cl.gI.3}
If $a\notin P^n_{\min}\cup P^n_{\max}$ and $sD^nf$ is decreasing on $I_L$, and increasing on $I_R$, then $D^{n+1}g_{I_k}(a)=0$, $D^{n+2}g_{I_k}(a)=s$ and $D^jg_{I_k}(a)=0$, $j > n+2$.

\item\label{cl.gI.3a}
If $a= e_{\min}\in P^n_{\min}$, $s_{\min}=-s$ and $sD^nf$ is increasing on $I_R$, or $a= e_{\max}\in P^n_{\max}$, $s_{\max}=s$ and $sD^nf$ is decreasing on $I_L$, then $D^{n+1}g_{I_k}(a)=0$, $D^{n+2}g_{I_k}(a)=s$ and $D^jg_{I_k}(a)=0$, $j > n+2$.

\item\label{cl.gI.4}
$|D^jg_I(x)-D^jf(x)|<\ep(x)/2$, $x\in I$, $j=0,\dots,n$, when $D^nf$ is not constant on $I$, in particular when $I$ is bounded

\item\label{cl.gI.5}
$|D^ng_I(x)-D^nf(x)|<\ep_0/2$, $x\in I$, when $D^nf$ is constant on $I$

\item\label{cl.gI.6}
If $e_{\max}\in P^n_{\max}$ with $I=[e_{\max},\infty)$, or $e_{\min}\in P^n_{\min}$ with $I=(-\infty,e_{\min}]$, then $D^jg_I$ has the same monotonicity as $D^jf$ on $I$, $j=0,\dots,n$.
\end{enumerate}
\end{claim}

\begin{proof}
For $j=0,\dots,n$, $D^jf$ is piecewise monotone since $D^nf$ is, by Proposition \ref{p:f.not.piecewise.monotone.implies.Df.also}.

Note that if $a<b$ and $E\cap (a,b)=\e$, then for $j=0,\dots,n$, $D^jf$ is monotone on $(a,b)$ by Proposition \ref{p:not piecewise monotone} since it has no turning platforms on $(a,b)$ by the assumption that $E$ has a point on each turning platform. Hence, $D^jf$ is monotone on $[a,b]$ by continuity.
We consider three cases.

\m

\n Case 1. $I=[a,b]$ with $a<b$ adjacent elements of $E$.

\m

As noted above, for $j=0,\dots,n$, $D^jf$ is monotone on $I$.
Apply Proposition \ref{p:approx.n.der.incr.minus} with $r=1$ and $s\in\{1,-1\}$ chosen so that $sD^nf$ is increasing on $I$, taking for $\ep$ the value
$\ep_1=\inf\{\ep(x)/2:a\leq x\leq b\}$. This produces a $\de>0$.
Let $\al_i=D^if(a)$, $\beta_i=D^if(b)$ for $j=0,\dots,n$.
Define $\al_j$, $\beta_j$ for $j>n$ as follows:

\smallskip

\n When $a,b\notin P^n_{\min}\cup P^n_{\max}$:
\begin{enumerate}[{\upshape\arabic*.}] 
\item\label{al.beta.1}
if $sD^nf$ is increasing on $[a^-,a]$, let $\al_{n+1}=s$, $\al_j=0$ for $j>n+1$;

\item\label{al.beta.2}
if $sD^nf$ is decreasing on $[a^-,a]$, let $\al_{n+1}=0$, $\al_{n+2}=s$, $\al_j=0$ for $j>n+2$;

\item\label{al.beta.5}
if $sD^nf$ is increasing on $[b,b^+]$, let $\beta_{n+1}=s$, $\beta_j=0$ for $j>n+1$;

\item\label{al.beta.6}
if $sD^nf$ is decreasing on $[b,b^+]$, let $\beta_{n+1}=0$, $\beta_{n+2}=-s$, $\beta_j=0$ for $j>n+2$.
\end{enumerate}
When $a=e_{\min}\in P^n_{\min}$:
\begin{enumerate}[{\upshape\arabic*.}] 
\setcounter{enumi}{4}
\item\label{al.beta.3}
if $s_{\min} = s$, let $\al_{n+1}=s$, $\al_j=0$ for $j>n+1$;

\item\label{al.beta.4}
if $s_{\min}=-s$, let $\al_{n+1}=0$, $\al_{n+2}=s$, $\al_j=0$ for $j>n+2$.
\end{enumerate}
When $b=e_{\max}\in P^n_{\max}$:
\begin{enumerate}[{\upshape\arabic*.}] 
\setcounter{enumi}{6}
\item\label{al.beta.7}
if $s_{\max}=s$, let $\beta_{n+1}=s$, $\beta_j=0$ for $j>n+1$;

\item\label{al.beta.8}
if $s_{\max}=-s$, let $\beta_{n+1}=0$, $\beta_{n+2}=-s$, $\beta_j=0$ for $j>n+2$.
\end{enumerate}
The conditions on the $\al_i$ and $\beta_i$ in Proposition \ref{p:approx.n.der.incr.minus} are then satisfied, so we get a $C^\infty$ function $g_I\colon I\to\R$ satisfying
\begin{enumerate}[label={(a\upshape\arabic*)}]
\item\label{1.1}
$D^jg_I(a)=\al_j$ and $D^jg_I(b)=\beta_j$, $j=0,1,\dots$

\item\label{1.2}
$D^{n+1}g_I(x)\not=0$, $a<x<b$

\item\label{1.3}
$|D^jg_I(x)-D^jf(x)|<\ep_1$, $j=0,\dots,n$, $a\leq x\leq b$

\item\label{1.4}
$D^jg_I(x)\not=0$, $j=1,\dots,n$, $a<x<b$
\end{enumerate}

Clause \ref{cl.gI.1} of the claim follows from \ref{1.2} and \ref{1.4}. Clauses \ref{cl.gI.1a}, \ref{cl.gI.2}, \ref{cl.gI.2a}, \ref{cl.gI.3} and \ref{cl.gI.3a} for $a$ and $I = I_R(a)$, and for $b$ and $I = I_L(b)$, follow from \ref{1.1} and the definitions of $\al_j$ and $\beta_j$. This is mostly immediate from these definitions. Note that 4.\ and 8.\ of the definition of $\beta_j$, together with \ref{1.1}, give parts of clauses \ref{cl.gI.3} and \ref{cl.gI.3a}. Writing the relevant parts in terms of $t\in\{1,-1\}$ and $b$ instead of $s$ and $a$, respectively, they state that
\begin{itemize}
\item
If $b\notin P^n_{\max}$ and $tD^nf$ is decreasing on $I=[a,b]=I_L(b)$, and increasing on $I_R(b)$, then $D^jg_{I}(b)=D^jf(b)$, $j=0,\dots,n$, $D^{n+1}g_{I}(b)=0$, $D^{n+2}g_{I}(b)=s$ and $D^jg_{I}(b)=0$, $j > n+2$.

\item
If $b = e_{\max}\in P^n_{\max}$, $tD^nf$ is decreasing on $I=[a,b]=I_L(b)$ and $s_{\max}=t$, then $D^jg_{I}(b)=D^jf(b)$, $j=0,\dots,n$, $D^{n+1}g_{I}(b)=0$, $D^{n+2}g_{I}(a)=t$ and $D^jg_{I_k}(a)=0$, $j > n+2$.
\end{itemize}
In both cases, under 4.\ or 8., respectively, the assumptions are satisfied when we take $t=-s$ and then \ref{1.1} gives the desired conclusion.

\m

\n Case 2. $I=[e_{\max},\infty)$ with $P^n_{\max}\not=\e$ and $e_{\max}\notin P^n_{\max}$, or $I=(-\infty,e_{\min}]$ with $P^n_{\min}\not=\e$ and $e_{\min}\notin P^n_{\min}$.

\m

Set $r=1$ under the first set of assumptions, $r=-1$ under the second set. For $r\in\{1,-1\}$, let $(\al,P^n) = (\al_r,P^n_r)$ be given by
\begin{itemize}
\item
$r=1$: $\al=e_{\max}$ and $P^n=P^n_{\max}$.

\item
$r=-1$: $\al = re_{\min}$ and $P^n = rP^n_{\min}$.
\end{itemize}
Then $r[\al,\infty)$ is the interval $[e_{\max},\infty)$ or $(-\infty,e_{\min}]$, depending on whether $r=1$ or $r=-1$. Write $P^n=[v,\infty)$.
Since $\al\notin P^n$ by assumption, we have $\al<v$.
Since $D^nf$ is constant on $r[v,\infty)$, $f$ is a polynomial on this interval. If $f$ is nonzero of degree $m$ then $m\leq n$ and we may choose $v_0\in [v,\infty)$ so that $D^jf(rv_0)\not=0$, $j=0,\dots,m$. If $f=0$ on $P^n$, take $v_0=v$.
Write $I_1=[\al,v_0]$, $I_2=[v_0,\infty)$.
There are no points of $E$ in the interior of $I_1$ since $E\cap rI_1 = \{r\al\}$, so each $D^jf$, $j=0,\dots,n$, is monotone on $rI_1$.

Let $s\in\{1,-1\}$ be such that $sD^nf$ is increasing on $rI_1$.
Since $D^nf$ is monotone and not constant between $r\al$ and $rv$, the constant value $k$ of $D^nf$ on $rP^n$ satisfies $D^nf(r\al)\not = k$.
We have that $x\mapsto srD^nf(rx)$ is increasing on $I_1$, so
\begin{equation}\label{eq:sr}
srD^nf(r\al) < srD^nf(rv_0) = srk.
\end{equation}
Apply Proposition \ref{p:approx.n.der.incr} to $f$ on $rI_1$ with $\ep_2=\inf\{\ep(rx)/2:\al \leq x\leq v_0\}$ taking the place of $\ep$ to get $\de>0$ having the properties stated in that proposition. 
Fix a positive number $b$ so that
\[
b<\min(\de/2,\ep(rv_0)/2). 
\]
Let $\beta_0=k-srb$, $\beta_1=s$, $\beta_j=0$ for $j>1$. Theorem \ref{t:approx.a.poly.on.platform} applies, taking for $\ep(x)$ the function $\min(\de/2,\ep(x)/2)$, and yields a $C^\infty$ function $g_{I_2}\colon rI_2\to\R$ satisfying

\begin{enumerate}[label={(b\upshape\arabic*)}]
\item\label{2.1}
$D^{n+1}g_{I_2}(rx) \not= 0$ for $x>v_0$

\item\label{2.2}
$D^ng_{I_2}(rv_0) = k - srb$, $D^{n+1}g_{I_2}(rv_0)=s$, $D^jg_{I_2}(rv_0)=0$, $j\geq n+2$.

\item\label{2.3}
$|D^jg_{I_2}(x)-D^jf(x)|<\min(\de/2,\ep(x)/2)$, $x\in rI_2$, $x\not=rv_0$, $j=0,\dots,n$

\item\label{2.4}
$s (-r)^{n+j+1}D^{j}g_{I_2}(x) > s(-r)^{n+j+1}D^{j}f(x)$, for all $x\in rI_2$, $j=0,\dots,n$.

\item\label{2.5}
$D^{j}g_{I_2}(rx)\not=0$ for $x\geq v_0$, $j=1,\dots,n$
\end{enumerate}

\begin{claim}\label{claim:not.opposite.sign}
When $f$ is not the zero polynomial on $rI$, and $m<j\leq n$, $D^jg_{I_2}(rv_0)$ and $D^jf(r\al)$ have the same nonzero sign. The same is true when $f=0$ on $rI$ for all $j=0,\dots,n$.
\end{claim}

\begin{proof}
Suppose $f$ is not the zero polynomial on $rI$. When $m<j\leq n$, we have $D^jf(rv_0) = 0$ since $f$ is a polynomial of degree $m$ on $rP^n$. In particular, $k=D^nf(rv_0) = 0$. The claim is then true for $j=n$ because $D^ng_{I_2}(rv_0) = k - srb = -srb$ has sign $-sr$, and from (\ref{eq:sr}) we see that when $k=0$, $D^nf(r\al)$ also has sign $-sr$. By \ref{2.4} it follows that when $m<j\leq n$, the sign of $D^jg_{I_2}(rv_0)$ is $s(-r)^{n+j+1}$. If $m<j<j+1\leq n$ and the sign of $D^{j+1}f(r\al)$ is $s(-r)^{n+j+2} = s(-r)^{n+j}$ then we verify that the sign of $D^{j}f(r\al)$ is $s(-r)^{n+j+1}$.

Note first that on the interval $r[\al,v_0]$, $D^jf$ is monotone but not constant, so since $D^jf(rv_0)=0$, we have $D^jf(r\al)\not=0$. If $r=1$ then consider the case $D^{j+1}f(\al)>0$. In this case, $D^{j+1}f$ decreases on $[\al,v_0]$ to its final value $0$, so it takes values $\geq 0$. Thus, $D^jf$, whose derivative it is, is an increasing function with final value $0$, so $D^{j}f(\al)<0$. Similarly, if $D^{j+1}f(\al)<0$ then $D^{j}f(\al)>0$. If $r=-1$, then consider the case $D^{j+1}f(-\al)>0$. In this case, $D^{j+1}f$ increases on $[-v_0,-\al]$ from its initial value of $0$, so it takes values $\geq 0$. Thus, $D^jf$, whose derivative it is, is also an increasing function with initial value $0$, so $D^{j}f(-\al)>0$. Similarly, if $D^{j+1}f(-\al)<0$ then $D^{j}f(-\al)<0$. In both settings, the sign of $D^{j}f(r\al)$ is $(-r)$ times that of $D^{j+1}f(r\al)$, so it is $s(-r)^{n+j+1}$.

Similarly for the case $f=0$ on $rI$.
\end{proof}

We define $I_0=[\beta,\al]=[\beta_r,\al_r]$ so that $rI_0\in\I$ is the neighbor of $rI_1$ on the opposite side from $rI_2$, as follows.
\begin{itemize}
\item
If $r=1$ then $I_0 = rI_0 = [\beta,\al] = [e^-_{\max},e_{\max}]$.

\item
If $r=-1$ then $rI_0 = [r\al,r\beta] = [e_{\min},e^+_{\min}]$.
\end{itemize}
Since $rI_0\in\I$ is bounded, Case 1 provides the function $g_{rI_0}$.
Write
\[
\mu_j = D^jg_{rI_0}(r\al),\ \ \nu_j=D^jg_{I_2}(rv_0),\ \ j=0,1,\dots.
\]
These satisfy the requirements for obtaining the $C^\infty$ function given by Proposition \ref{p:approx.n.der.incr}, as we now verify.
\begin{enumerate}[label={--},leftmargin=0.5cm]
\item
$|\mu_j - D^jf(r\al)| = |D^jg_{I_0}(r\al) - D^jf(r\al)| = 0$, $j=0,\dots,n$.

\item
$|\nu_j - D^jf(rv_0)| = |D^jg_{I_2}(rv_0) - D^jf(rv_0)| \leq \de/2 < \de$
(by \ref{2.3} and continuity of $D^jg_{I_2}$, $D^jf$).

\item
By \ref{2.2}, the least $j>n$ for which $\nu_j = D^jg_{I_2}(rv_0)\not=0$ is $j=n+1$. We have $s(-r)^{n+j+1}\nu_j = s\nu_{n+1} = ss = 1 > 0$.

\item
Since $r\al\notin P^n_{\min}\cup P^n_{\max}$, by \ref{cl.gI.2} and \ref{cl.gI.3}, 
the least $j>n$ for which $\mu_j = D^jg_{rI_0}(r\al)\not=0$ is either $n+1$ or $n+2$.
We have $j=n+1$ when $sD^nf$ is increasing on $rI_0$ and in that case, by \ref{cl.gI.2}, $\mu_{n+1}=s$, so $sr^{n+j+1}\mu_{j} = s\mu_{n+1} = ss = 1 > 0$.
We have $j=n+2$ when $sD^nf$ is decreasing on $rI_0$ and in that case, by \ref{cl.gI.3} with $sr$ in the place of $s$, $\mu_{n+2}=sr$,%
\footnote{When $r=1$, $rI_0=I_L$ $(=I_L(r\al))$ and $sD^nf$ decreases on $I_L$ and increases on $I_R$, so by \ref{cl.gI.3}, $D^{n+2}g_{rI_0}(r\al)=s=sr$. When $r=-1$, $rI_0=I_R$ and $sD^nf$ increases on $I_L$ and decreases on $I_R$, so $srD^nf$ decreases on $I_L$ and increases on $I_R$, giving by \ref{cl.gI.3}, $D^{n+2}g_{rI_0}(r\al)=sr$.}
so $sr^{n+j+1}\mu_j = sr(sr) = 1 > 0$.
\end{enumerate}
Thus, by the choice of $\de$, there is a $C^\infty$ function $g_{I_1}\colon rI_1\to\R$ such that the following holds.

\begin{enumerate}[label={(c\upshape\arabic*)}]
\item\label{3.1}
$D^jg_{I_1}(r\al)=\mu_j$ and $D^jg_{I_1}(rv_0)=\nu_j$, $j=0,1,\dots$, and in particular,

\n $D^jg_{I_1}(r\al)=\mu_j=D^jg_{rI_0}(r\al)=D^jf(r\al)$ for $j=0,\dots,n$.

\item\label{3.2}
$D^{n+1}g_{I_1}(rx)\not=0$, $\al<x<v_0$

\item\label{3.3}
$|D^jg_{I_1}(rx)-D^jf(rx)|<\ep(rx)/2$, $\al\leq x\leq v_0$, $j=0,\dots,n$

\item\label{3.4}
$D^jg_{I_1}(rx)\not=0$, $j=1,\dots,n$, $\al<x<v_0$
\end{enumerate}

In \ref{3.1}, when $j=0,\dots,n$, we have $\mu_j=D^jg_{rI_0}(r\al)=D^jf(r\al)$ by \ref{cl.gI.2} or \ref{cl.gI.3} for $r\al$.
For \ref{3.4}, we use clause \ref{2.9.5} of Proposition \ref{p:approx.n.der.incr}. For that we need to verify for $j=1,\dots,n$, the two conditions
\begin{enumerate}[(a)]
\item
if $D^jf(rv_0)=0$ then $\nu_jD^jf(r\al)\geq 0$, and

\item
if $D^jf(r\al)=0$ then $\mu_jD^jf(rv_0)\geq 0$.
\end{enumerate}
Condition (b) holds since its hypothesis is $\mu_j=0$. For (a), if $f$ is not the zero polynomial on $rI$, we have for $j=1,\dots,m$ that $D^jf(rv_0)\not=0$, while by Claim \ref{claim:not.opposite.sign}, $\nu_j=D^jg_{I_2}(rv_0)$ has the same sign as $D^jf(r\al)$ for $j=m+1,\dots,n$. Claim \ref{claim:not.opposite.sign} also takes care of the case where $f=0$ on $rI$.

The function $g_I=g_{I_1}\cup g_{I_2}$ satisfies \ref{cl.gI.1}, \ref{cl.gI.1a}, \ref{cl.gI.2}, \ref{cl.gI.3} and \ref{cl.gI.4} with respect to $r\al$. \ref{cl.gI.1} follows from \ref{3.2}, \ref{3.4}, \ref{2.1}, \ref{2.2}, \ref{2.5}. By \ref{3.1}, \ref{cl.gI.1a}, \ref{cl.gI.2} and \ref{cl.gI.3} hold for $g_I$ at $r\al$ since they hold for $g_{rI_0}$. \ref{cl.gI.4} holds by \ref{3.3}, \ref{2.3}. The other clauses do not apply.

\m

\n Case 3.  $I=[e_{\max},\infty)$ with $e_{\max}\in P^n_{\max}$, or $I=(-\infty,e_{\min}]$ with $e_{\min}\in P^n_{\min}$.

\m

For $r\in\{1,-1\}$, define $\al$, $P^n=[v,\infty)$, $I_0=[\beta,\al]$ and $k$ as in Case 2. For $r \in\{1,-1\}$, let $s_r=s_{\max}$ if $r=1$, and $s_r=s_{\min}$ if $r=-1$. For $j=1,2,\dots$, define $\beta_j$ as follows.
\begin{itemize}
\item
If $s_rD^nf$ is increasing on $rI_0$, then $\beta_{1}=s_r$, $\beta_j=0$ for $j>1$;

\item
if $s_rD^nf$ is decreasing on $rI_0$, then $\beta_{1}=0$, $\beta_{2}=rs_r$, $\beta_j=0$ for $j>2$.
\end{itemize}
In the first case, the least $j$ for which $\beta_j\not=0$ is $j=1$. We have $s_rr^{j+1}\beta_j = r^{2} = 1>0$.
In the second case, the least $j$ for which $\beta_j\not=0$ is $j=2$. We have $s_rr^{j+1}\beta_j = r^{4} = 1>0$.
By Theorem \ref{t:approx.a.poly.on.platform.2}, there is a $C^\infty$
function $g_I\colon I\to\R$ such that the following hold.
\begin{enumerate}[label={(d\upshape\arabic*)}]
\item\label{4.1}
$s_r D^{n+1}g_I(x)>0$, $x\in I$, $x\not=r\al$.

\item\label{4.2}
$D^{n+j}g_I(r\al)=\beta_j$, $j=1,2,\dots$.

\item\label{4.3}
$D^jg_I(r\al) = D^jf(r\al)$, $j=0,\dots,n$.

\item\label{4.4}
$|D^ng_I(x)-D^nf(x)| < \ep_0/2$, $x\in I$.

\item\label{4.5}
$D^{j}g_I(x)\not=0$ for all $x\in I\sm\{ra\}$, $j=1,\dots,n$.

\item\label{4.6}
$D^jg_I$ has the same monotonicity as $D^jf$ on $I$, $j=0,\dots,n$.
\end{enumerate}
\ref{cl.gI.1} holds by \ref{4.1}, \ref{4.5}.
\ref{cl.gI.1a} holds by \ref{4.3}.
\ref{cl.gI.2a} and \ref{cl.gI.3a} hold by \ref{4.2} and the definition of $\beta_j$. (The assumption of \ref{cl.gI.3a} when $a=e_{\min}$ or $a=e_{\max}$ holds for $s=rs_r$.)
\ref{cl.gI.5} and \ref{cl.gI.6} hold by \ref{4.5} and \ref{4.6}. The other clauses do not apply.

This completes the proof of the claim.
\end{proof}

Now that we are given a function $g_I$ for each $I\in\I$, we define a function $g_0\colon\R\to\R$ by $g_0(x)=g_I(x)$ for $x\in I\in\I$. It is readily verified that $g_0$ is well-defined and $C^\infty$, and by Claim \ref{claim:g_I.properties} we have the following properties.
\begin{enumerate}[label={(e\upshape\arabic*)}]
\item\label{5.1}
$D^jg_0(a)=D^jf(a)$ when $a\in E$, $j=0,\dots,n$

\item\label{5.2}
$D^jg_0(x)\not=0$ for $x\not\in E$, $1 \leq j \leq  n+1$.

\item\label{5.3}
$D^{n+1}g_0(a)\not=0$ when $a\in E\sm (P^n_{\min}\cup P^n_{\max})$ and $D^nf$ is monotone on $(a^-,a^+)$

\item\label{5.4}
$D^{n+1}g_0(a)=0$ and $D^{n+2}g_0(a)\not=0$ when $a\in E\sm (P^n_{\min}\cup P^n_{\max})$ and $D^nf$ is not monotone on $(a^-,a^+)$

\item\label{5.4a}
When $a\in E\cap (P^n_{\min}\cup P^n_{\max})$, we have either $D^{n+1}g_0(a)\not=0$ or $D^{n+2}g_0(a)\not=0$.

\item\label{5.5}
$|D^ng_0(x)-D^nf(x)|<\ep_0/2$, $x\in \R$

\item\label{5.6}
$|D^jg_0(x)-D^jf(x)|<\ep(x)/2$ when $x\notin W_n$, $j=0,\dots,n$.

\item\label{5.7}
If $e_{\max}\in P^n_{\max}$ with $I=[e_{\max},\infty)$, or $e_{\min}\in P^n_{\min}$ with $I=(-\infty,e_{\min}]$, then $D^jg_0$ has the same monotonicity as $D^jf$ on $I$, $j=0,\dots,n$.
\end{enumerate}
Note that by \ref{5.2}, \ref{5.3}, \ref{5.4}, \ref{5.4a}, $D^{n+1}g_0$ has no flat points as a $C^{n+2}$ function since we have $D^{n+2}g_0(a)\not=0$ whenever $D^{n+1}g_0(a)=0$.

By Theorem \ref{t:piecewise.monotone.interpolation}, taking $n$ and $m$ to be $n+1$ and $n+2$, respectively, $T=E$, and $U_i=\e$ for $i=1,\dots,n+2$, we get an entire function $g$ such that $g(\R)\sq\R$ and for all $x\in \R$,
\begin{enumerate}[label={(f\upshape\arabic*)}]
\item\label{ff.1}
$|D^ig(x) - D^ig_0(x)|<\varepsilon(x)/2$, $i=0,\dots,n+2$.

\item\label{ff.2}
$D^ig(x) = D^ig_0(x)$, $x\in E$, $i=0,\dots,n+2$.

\item\label{ff.3}
$D^kg(x)$ has the same sign as $D^kg_0(x)$, $x\in \R$, $k=0,\dots,n+1$.
\end{enumerate}

We then get \ref{p:C.n.Hoischen:5} from \ref{5.1} and \ref{ff.2}, we get \ref{p:C.n.Hoischen:1a} from \ref{5.2} and \ref{ff.3}, and we get
\ref{p:C.n.Hoischen:2} from \ref{5.3} and \ref{ff.2}. For \ref{p:C.n.Hoischen:3} and \ref{p:C.n.Hoischen:4}, use \ref{5.5}, \ref{5.6} and \ref{ff.1}.
For \ref{p:C.n.Hoischen:6}, let $j=0,\dots,n$.
For bounded intervals $I\in\I$, we have that $D^jf$ is monotone but not constant on $I$. From \ref{p:C.n.Hoischen:1a}, $D^jg$ is also monotone on $I$, and by \ref{p:C.n.Hoischen:5} it agrees with $D^jf$ at the endpoints, so $D^jg$ and $D^jf$ have the same monotonicity on $I$.

For an unbounded interval $I\in\I$ with endpoint $a\notin P^n_{\min}\cup P^n_{\max}$, we have again that $D^jf$ is monotone but not constant on $I$, $D^jg$ is monotone on $I$, and $D^jg(a) = D^jf(a)$. We have the point $b_{I,j}$ satisfying (\ref{eq:bIj}). Then \ref{p:C.n.Hoischen:4} forces $D^jg(b(I,j)) - D^jg(a)$ to have the same sign as $D^jf(b(I,j)) - D^jf(a)$ and hence forces $D^jg$ to have the same monotonicity as $D^jf$ on $I$.

For an unbounded interval $I\in\I$ with endpoint $a\in P^n_{\min}\cup P^n_{\max}$, we have by \ref{5.7} that $D^jg_0$ has the same monotonicity as $D^jf$ on $I$. By \ref{ff.3}, $D^jg_0$ and $D^jg$ have the same monotonicity on $I$, and therefore so do $D^jg$ and $D^jf$.
\end{proof}

\begin{rem}\label{r:only.Q.minus}
If each of $P^n_{\min}$, $P^n_{\max}$ is either empty or contains a point of $E$, then Case 2 in the proof of Claim \ref{claim:g_I.properties} is not needed and the proof only uses $(Q_n^-)$ (via the appeal to Proposition \ref{p:approx.n.der.incr.minus}) rather than $(Q_n)$.
\end{rem}

\begin{rem}\label{r:f.not.cst.on.component.plus.zero.limits}
When $\lim_{x\to\pm\infty}\ep(x)=0$, \ref{p:C.n.Hoischen:4} shows that $\lim_{x\to \infty}(D^ng(x)-D^nf(x))=0$ unless $E\cap P^n_{\max}\not=\e$, and that $\lim_{x\to -\infty}(D^ng(x)-D^nf(x))=0$ unless $E\cap P^n_{\min}\not=\e$.
When $E\cap P^n_{\max}=\{q\}\not=\e$, then on $[q,\infty)$, $D^nf$ is constant and $D^ng$ is monotone (by \ref{p:C.n.Hoischen:2}) and bounded (by \ref{p:C.n.Hoischen:3}), so $\lim_{x\to \infty}(D^ng(x)-D^nf(x))$ exists.
Similarly, when $E\cap P^n_{\min}\not=\e$, $\lim_{x\to -\infty}(D^ng(x)-D^nf(x))$ exists.
\end{rem}

\begin{rem}\label{r:C.n.Hoischen}
The theorem gives no explicit bound for $|D^jg(x)-D^jf(x)|$, $0\leq j< n$, when $x\in W_n$. However using the fact from \ref{p:C.n.Hoischen:5} that $D^jg(x)=D^jf(x)$ when $x=e_{\min}$ or $e_{\max}$, $j=0,\dots,n$, we easily get inductively that for $k=0,\dots,n$,
\begin{itemize}
\item
$|D^{n-k}g(x)-D^{n-k}f(x)|<\ep_0(x-e_{\max})^k/k!$, $e_{\max}\in P^n_{\max}$, $x>e_{\max}$

\item
$|D^{n-k}g(x)-D^{n-k}f(x)|<\ep_0(e_{\min}-x)^k/k!$, $e_{\min}\in P^n_{\min}$, $x>e_{\min}$
\end{itemize}
\end{rem}

In Theorem \ref{t:C.infty.Hoischen.C.infty.fcts} \ref{i:t:C.infty.Hoischen:1}, we required for the case $n=0$ that $D^{n+1}f(x)\not=0$ when $x\in E$ is not a turning point of $D^nf$, including when $x=e_{\max}\in P^n_{\max}$. That has the advantage, writing $a$ for the predecessor of $e_{\max}$ in $E$, that $D^ng$ will be monotone on $[a,\infty)$, just as $D^nf$ is. The following example illustrates why that is a bad strategy when $n>0$ if we want $D^jg$ comonotone with $D^jf$ for all $j=0,\dots,n$.

\begin{exmp}\label{exmp:limit.of.2}
Take $n=2$ and let $f$ be the $C^n$ function given by $f(x)=x^3-x$, $x\leq 0$, $f(x)=-x$, $x\geq 0$.
$Df$ and $D^2f$ are both monotone, and $f$ has one turning point at $-1/\sqrt 3$. Take $E=\{a,e_{\max}\}$, where $a=-1/\sqrt{3}$, $e_{\max}=0$, and note that $0$ is not a turning point for $D^2f$. If we ask that $D^{n+1}g(x) = D^3g(x) > 0$ for all $x\in(a,\infty)$, then because $D^2g(0)=0$ and $D^2g$ is increasing on $[0,\infty)$, we must have that $\lim_{x\to\infty}Dg(x)=\lim_{x\to\infty}-1+\int_0^xD^2g(t)\,dt=\infty$ and hence there is an $a>0$ such that $Dg(a)=0$. Thus, $Dg$ changes sign on $(0,\infty)$ and hence $g$ is not monotone on this interval. Moreover the monotonicity of $g$ is different from that of $f$ on $(a,\infty)$, so $f$ and $g$ are not comonotone for any witnessing set. See Figure 1. The approximations provided by Theorem \ref{t:C.n.Hoischen.C.n.fcts} do not have $D^2g$ monotone on $(a,\infty)$, but they make $D^jg$ comonotone with $D^jf$ for $j=0,1,2$, with $E$ as witnessing set, as in Figure 2.

\[
\begin{tikzpicture}
\def\u{0.5cm}
\def\uf{0.5} 
\def\tt{0.1}
\begin{scope}
\begin{scope}
% axes
  	\draw (-2*\u,0) -- (2*\u,0);
	\draw (0,-4*\u) -- (0,2*\u);
%tick marks
	\draw (-1*\u,-\tt*\u) -- ++(0,2*\tt*\u);
	\draw (-\tt*\u,-1*\u) -- ++(2*\tt*\u,0);
	\node[below] at (-1*\u,0) {\tiny $-1$};
	\node[left] at (0,-1*\u) {\tiny $-1$};
%curve f
    \draw [domain={-1.7*\uf}:0, samples=50, smooth]
        plot (\x,{\uf*((\x/\uf)^3 - (\x/\uf))});
	\draw (0,0) -- (2*\u,-2*\u);
	\fill (0,0) circle [radius=1.2pt];
	\fill (-0.577*\u,0.385*\u) circle [radius=1.2pt];
	\node at (1*\u,-1.7*\u) {\sz $f$};
	\node at (1*\u,0.7*\u) {\sz $g$};
    \coordinate (start) at (-1.75*\u,-3.213*\u);
    \coordinate (p1) at ({(-1/sqrt(3))*\u},{(2/3^1.5)*\u});
    \coordinate (p2) at (0,0);
    \coordinate (end) at (2*\u,-0.2*\u);
    \coordinate (c1) at ($(start) + 0.5*(1*\u,7.67*\u)$);
    \coordinate (c2) at ($(p1) - 0.3*(1*\u,0*\u)$);
    \coordinate (c3) at ($(p1) + 0.3*(1*\u,0*\u)$);
    \coordinate (c4) at ($(p2) - 0.3*(1*\u,-1*\u)$);
    \coordinate (c5) at ($(p2) + 0.5*(1*\u,-1*\u)$);
    \coordinate (c6) at ($(end) - 0.5*(1*\u,1*\u)$);
    \draw[dashed,smooth] (start)
      .. controls (c1) and (c2) .. (p1)
      .. controls (c3) and (c4) .. (p2)
      .. controls (c5) and (c6) .. (end);
\end{scope}
\begin{scope}[xshift=5*\u]
% axes
	\draw (-2*\u,0) -- (2*\u,0);
	\draw (0,-4*\u) -- (0,2*\u);
%tick marks
	\draw (-1*\u,-\tt*\u) -- ++(0,2*\tt*\u);
	\draw (-\tt*\u,-1*\u) -- ++(2*\tt*\u,0);
	\node[below] at (-1*\u,0) {\tiny $-1$};
	\node[left] at (0,-1*\u) {\tiny $-1$};
%curve f
    \draw [domain={-1*\uf}:0, samples=50, smooth]
        plot (\x,{\uf*(3*(\x/\uf)^2 - 1)});
	\draw (0,-1*\u) -- ++(2*\u,0);
	\fill (0,-1*\u) circle [radius=1.2pt];
	\fill (-0.577*\u,0) circle [radius=1.2pt];
	\node at (1*\u,-1.7*\u) {\sz $Df$};
	\node at (1*\u,0.8*\u) {\sz $Dg$};
    \coordinate (start) at (-1.07*\u,2*\u);
    \coordinate (p1) at ({(-1/sqrt(3))*\u},0);
    \coordinate (p2) at (0,-1*\u);
    \coordinate (end) at (2*\u,1*\u);
    \coordinate (c1) at ($(start) + 0.05*(1*\u,-6*\u)$);
    \coordinate (c2) at ($(p1) - 0.1*(2.1*\u,{(-6/sqrt(3))*\u})$);
    \coordinate (c3) at ($(p1) + 0.15*(2.1*\u,{(-6/sqrt(3))*\u})$);
    \coordinate (c4) at ($(p2) - 0.45*(1*\u,0*\u)$);
    \coordinate (c5) at ($(p2) + 0.3*(1*\u,0*\u)$);
    \coordinate (c6) at ($(end) - 0.03*(1*\u,1*\u)$);
    \draw[dashed,smooth] (start)
      .. controls (c1) and (c2) .. (p1)
      .. controls (c3) and (c4) .. (p2)
      .. controls (c5) and (c6) .. (end);
\end{scope}
\begin{scope}[xshift=10*\u]
% axes
	\draw (-2*\u,0) -- (2*\u,0);
	\draw (0,-4*\u) -- (0,2*\u);
%tick marks
	\draw (-1*\u,-\tt*\u) -- ++(0,2*\tt*\u);
	\draw (-\tt*\u,-1*\u) -- ++(2*\tt*\u,0);
	\node[below] at (-1*\u,0) {\tiny $-1$};
	\node[left] at (0,-1*\u) {\tiny $-1$};
%curve f
	\draw (-0.667*\u,-4*\u) -- (0,0);
	\draw (0,0) -- ++(2*\u,0);
	\fill (0,0) circle [radius=1.2pt];
	\fill (-0.577*\u,-6*0.577*\u) circle [radius=1.2pt];
	\node at (1*\u,-0.6*\u) {\sz $D^2f$};
	\node at (1*\u,0.9*\u) {\sz $D^2g$};
    \coordinate (start) at (-0.64*\u,-4*\u);
    \coordinate (p1) at ({(-1/sqrt(3))*\u},{(-6/sqrt(3))*\u});
    \coordinate (p2) at ({(-0.666/sqrt(3))*\u},{(-6*0.666/sqrt(3))*\u});
    \coordinate (p3) at ({(-0.333/sqrt(3))*\u},{(-6*0.333/sqrt(3))*\u});
    \coordinate (p4) at (0,0);
    \coordinate (end) at (2*\u,0.3*\u);
    \coordinate (c1) at ($(start) + 0.05*(0.7*\u,6*\u)$);
    \coordinate (c2) at ($(p1) - 0.05*(0.5*\u,6*\u)$);
    \coordinate (c3) at ($(p1) + 0.1*(0.5*\u,6*\u)$);
    \coordinate (c4) at ($(p2) - 0.1*(1.7*\u,6*\u)$);
    \coordinate (c5) at ($(p2) + 0.1*(1.7*\u,6*\u)$);
    \coordinate (c6) at ($(p3) - 0.1*(0.7*\u,6*\u)$);
    \coordinate (c7) at ($(p3) + 0.2*(0.7*\u,6*\u)$);
    \coordinate (c8) at ($(p4) - 0.12*(1*\u,1*\u)$);
    \coordinate (c9) at ($(p4) + 0.2*(1*\u,1*\u)$);
    \coordinate (c10) at ($(end) - 0.03*(8*\u,1*\u)$);
    \draw[dashed,smooth] (start)
      .. controls (c1) and (c2) .. (p1)
      .. controls (c3) and (c4) .. (p2)
      .. controls (c5) and (c6) .. (p3)
      .. controls (c7) and (c8) .. (p4)
      .. controls (c9) and (c10) .. (end);
\end{scope}
\end{scope}
\node[below] at (current bounding box.south) {\sz Figure 1};
\end{tikzpicture}
\rule{1.5cm}{0cm}
\begin{tikzpicture}
\def\u{0.5cm}
\def\uf{0.5} 
\def\tt{0.05}
\begin{scope}
\begin{scope}
% axes
  	\draw (-2*\u,0) -- (2*\u,0);
	\draw (0,-4*\u) -- (0,2*\u);
%tick marks
	\draw (-1*\u,-\tt*\u) -- ++(0,2*\tt*\u);
	\draw (-\tt*\u,-1*\u) -- ++(2*\tt*\u,0);
	\node[below] at (-1*\u,0) {\tiny $-1$};
	\node[left] at (0,-1*\u) {\tiny $-1$};
%curve f
    \draw [domain={-1.7*\uf}:0, samples=50, smooth]
        plot (\x,{\uf*((\x/\uf)^3 - (\x/\uf))});
	\draw (0,0) -- (2*\u,-2*\u);
	\fill (0,0) circle [radius=1.2pt];
	\fill (-0.577*\u,0.385*\u) circle [radius=1.2pt];
	\node at (1.2*\u,-0.45*\u) {\sz $f$};
	\node at (1.2*\u,-2.5*\u) {\sz $g$};
    \coordinate (start) at (-1.75*\u,-3.213*\u);
    \coordinate (p1) at ({(-1/sqrt(3))*\u},{(2/3^1.5)*\u});
    \coordinate (p2) at (0,0);
    \coordinate (end) at (2*\u,-3*\u);
    \coordinate (c1) at ($(start) + 0.5*(1*\u,7.67*\u)$);
    \coordinate (c2) at ($(p1) - 0.3*(1*\u,0*\u)$);
    \coordinate (c3) at ($(p1) + 0.3*(1*\u,0*\u)$);
    \coordinate (c4) at ($(p2) - 0.3*(1*\u,-1*\u)$);
    \coordinate (c5) at ($(p2) + 0.5*(1*\u,-1*\u)$);
    \coordinate (c6) at ($(end) - 0.5*(1*\u,-2*\u)$);
    \draw[dashed,smooth] (start)
      .. controls (c1) and (c2) .. (p1)
      .. controls (c3) and (c4) .. (p2)
      .. controls (c5) and (c6) .. (end);
\end{scope}
\begin{scope}[xshift=5*\u]
% axes
	\draw (-2*\u,0) -- (2*\u,0);
	\draw (0,-4*\u) -- (0,2*\u);
%tick marks
	\draw (-1*\u,-\tt*\u) -- ++(0,2*\tt*\u);
	\draw (-\tt*\u,-1*\u) -- ++(2*\tt*\u,0);
	\node[below] at (-1*\u,0) {\tiny $-1$};
	\node[left] at (0,-1*\u) {\tiny $-1$};
%curve Df
    \draw [domain={-1*\uf}:0, samples=50, smooth]
        plot (\x,{\uf*(3*(\x/\uf)^2 - 1)});
	\draw (0,-1*\u) -- ++(2*\u,0);
	\fill (0,-1*\u) circle [radius=1.2pt];
	\fill (-0.577*\u,0) circle [radius=1.2pt];
	\node at (1*\u,-0.5*\u) {\sz $Df$};
	\node at (1*\u,-1.9*\u) {\sz $Dg$};
    \coordinate (start) at (-1.07*\u,2*\u);
    \coordinate (p1) at ({(-1/sqrt(3))*\u},0);
    \coordinate (p2) at (0,-1*\u);
    \coordinate (end) at (2*\u,-1.7*\u);
    \coordinate (c1) at ($(start) + 0.05*(1*\u,-6*\u)$);
    \coordinate (c2) at ($(p1) - 0.1*(2.1*\u,{(-6/sqrt(3))*\u})$);
    \coordinate (c3) at ($(p1) + 0.15*(2.1*\u,{(-6/sqrt(3))*\u})$);
    \coordinate (c4) at ($(p2) - 0.45*(1*\u,0*\u)$);
    \coordinate (c5) at ($(p2) + 0.3*(1*\u,0*\u)$);
    \coordinate (c6) at ($(end) - 0.03*(1*\u,-1*\u)$);
    \draw[dashed,smooth] (start)
      .. controls (c1) and (c2) .. (p1)
      .. controls (c3) and (c4) .. (p2)
      .. controls (c5) and (c6) .. (end);
\end{scope}
\begin{scope}[xshift=10*\u]
% axes
	\draw (-2*\u,0) -- (2*\u,0);
	\draw (0,-4*\u) -- (0,2*\u);
%tick marks
	\draw (-1*\u,-\tt*\u) -- ++(0,2*\tt*\u);
	\draw (-\tt*\u,-1*\u) -- ++(2*\tt*\u,0);
	\node[below] at (-1*\u,0) {\tiny $-1$};
	\node[left] at (0,-1*\u) {\tiny $-1$};
%curve D2f
	\draw (-0.667*\u,-4*\u) -- (0,0);
	\draw (0,0) -- ++(2*\u,0);
	\fill (0,0) circle [radius=1.2pt];
	\fill (-0.577*\u,-6*0.577*\u) circle [radius=1.2pt];
	\node at (1*\u,0.6*\u) {\sz $D^2f$};
	\node at (1*\u,-0.8*\u) {\sz $D^2g$};
    \coordinate (start) at (-0.64*\u,-4*\u);
    \coordinate (p1) at ({(-1/sqrt(3))*\u},{(-6/sqrt(3))*\u});
    \coordinate (p2) at ({(-0.666/sqrt(3))*\u},{(-6*0.666/sqrt(3))*\u});
    \coordinate (p3) at ({(-0.333/sqrt(3))*\u},{(-6*0.333/sqrt(3))*\u});
    \coordinate (p4) at (0,0);
    \coordinate (end) at (2*\u,-0.3*\u);
    \coordinate (c1) at ($(start) + 0.05*(0.7*\u,6*\u)$);
    \coordinate (c2) at ($(p1) - 0.05*(0.5*\u,6*\u)$);
    \coordinate (c3) at ($(p1) + 0.1*(0.5*\u,6*\u)$);
    \coordinate (c4) at ($(p2) - 0.1*(1.7*\u,6*\u)$);
    \coordinate (c5) at ($(p2) + 0.1*(1.7*\u,6*\u)$);
    \coordinate (c6) at ($(p3) - 0.1*(0.7*\u,6*\u)$);
    \coordinate (c7) at ($(p3) + 0.2*(0.7*\u,6*\u)$);
    \coordinate (c8) at ($(p4) - 0.12*(1*\u,0*\u)$);
    \coordinate (c9) at ($(p4) + 0.05*(1*\u,0*\u)$);
    \coordinate (c10) at ($(end) - 0.7*(1*\u,0*\u)$);
    \draw[dashed,smooth] (start)
      .. controls (c1) and (c2) .. (p1)
      .. controls (c3) and (c4) .. (p2)
      .. controls (c5) and (c6) .. (p3)
      .. controls (c7) and (c8) .. (p4)
      .. controls (c9) and (c10) .. (end);
\end{scope}
\end{scope}
\node[below] at (current bounding box.south) {\sz Figure 2};
\end{tikzpicture}
\]

If we remove $0$ from the interpolation set $E$, then we get a comonotone approximation as in Figure 3. Note that on $[0,\infty)$, the approximations $D^jg$ are alternately above and below $D^jf$, as specified in Theorem \ref{t:approx.a.poly.on.platform} \ref{3.4.4}.
\[
\begin{tikzpicture}
\def\u{0.5cm}
\def\uf{0.5} 
\def\tt{0.05}
\begin{scope}
% axes
	\draw (-2*\u,0) -- (2*\u,0);
	\draw (0,-4*\u) -- (0,2*\u);
%tick marks
	\draw (-1*\u,-\tt*\u) -- ++(0,2*\tt*\u);
	\draw (-\tt*\u,-1*\u) -- ++(2*\tt*\u,0);
	\node[below] at (-1*\u,0) {\tiny $-1$};
	\node[left] at (0,-1*\u) {\tiny $-1$};
%curve f
    \draw [domain={-1.7*\uf}:0, samples=50, smooth]
        plot (\x,{\uf*((\x/\uf)^3 - (\x/\uf))});
	\draw (0,0) -- (2*\u,-2*\u);
	\fill (0,0) circle [radius=1.2pt];
	\fill (-0.577*\u,0.385*\u) circle [radius=1.2pt];
	\node at (1.2*\u,-1.8*\u) {\sz $g$};
	\node at (1.2*\u,-0.5*\u) {\sz $f$};
    \coordinate (start) at (-1.8*\u,-3.213*\u);
    \coordinate (p1) at ({(-1/sqrt(3))*\u},{(2/3^1.5)*\u});
    \coordinate (p2) at (0,-0.2*\u);
    \coordinate (end) at (2*\u,-2.07*\u);
    \coordinate (c1) at ($(start) + 0.05*(1*\u,7.67*\u)$);
    \coordinate (c2) at ($(p1) - 0.8*(1*\u,0*\u)$);
    \coordinate (c3) at ($(p1) + 0.3*(1*\u,0*\u)$);
    \coordinate (c4) at ($(p2) - 0.1*(1*\u,-0.6*\u)$);
    \coordinate (c5) at ($(p2) + 0.3*(1*\u,-0.6*\u)$);
    \coordinate (c6) at ($(end) - 0.03*(1*\u,-1*\u)$);
    \draw[dashed,smooth] (start)
      .. controls (c1) and (c2) .. (p1)
      .. controls (c3) and (c4) .. (p2)
      .. controls (c5) and (c6) .. (end);
\end{scope}
\begin{scope}[xshift=5*\u]
% axes
	\draw (-2*\u,0) -- (2*\u,0);
	\draw (0,-4*\u) -- (0,2*\u);
%tick marks
	\draw (-1*\u,-\tt*\u) -- ++(0,2*\tt*\u);
	\draw (-\tt*\u,-1*\u) -- ++(2*\tt*\u,0);
	\node[below] at (-1*\u,0) {\tiny $-1$};
	\node[left] at (0,-1*\u) {\tiny $-1$};
%curve f
    \draw [domain={-1*\uf}:0, samples=50, smooth]
        plot (\x,{\uf*(3*(\x/\uf)^2 - 1)});
	\draw  (0,-1*\u) -- ++(2*\u,0);
	\fill (0,-1*\u) circle [radius=1.2pt];
	\fill (-0.577*\u,0) circle [radius=1.2pt];
	\node at (1*\u,-1.7*\u) {\sz $Df$};
	\node at (1*\u,-0.5*\u) {\sz $Dg$};
    \coordinate (start) at (-0.9*\u,2*\u);
    \coordinate (p1) at ({(-1/sqrt(3))*\u},0);
    \coordinate (p2) at (0,-0.85*\u);
    \coordinate (end) at (2*\u,-0.95*\u);
    \coordinate (c1) at ($(start) + 0.05*(1*\u,-7.67*\u)$);
    \coordinate (c2) at ($(p1) - 0.2*(1*\u,-6*\u)$);
    \coordinate (c3) at ($(p1) + 0.12*(1*\u,-6*\u)$);
    \coordinate (c4) at ($(p2) - 0.5*(1*\u,-0.2*\u)$);
    \coordinate (c5) at ($(p2) + 0.2*(1*\u,-0.2*\u)$);
    \coordinate (c6) at ($(end) - 0.8*(1*\u,0*\u)$);
    \draw[dashed,smooth] (start)
      .. controls (c1) and (c2) .. (p1)
      .. controls (c3) and (c4) .. (p2)
      .. controls (c5) and (c6) .. (end);
\end{scope}
\begin{scope}[xshift=10*\u]
% axes
	\draw (-2*\u,0) -- (2*\u,0);
	\draw (0,-4*\u) -- (0,2*\u);
%tick marks
	\draw (-1*\u,-\tt*\u) -- ++(0,2*\tt*\u);
	\draw (-\tt*\u,-1*\u) -- ++(2*\tt*\u,0);
	\node[below] at (-1*\u,0) {\tiny $-1$};
	\node[left] at (0,-1*\u) {\tiny $-1$};
%curve f
	\draw (-0.667*\u,-4*\u) -- (0,0);
	\draw (0,0) -- ++(2*\u,0);
	\fill (0,0) circle [radius=1.2pt];
	\fill (-0.577*\u,-6*0.577*\u) circle [radius=1.2pt];
	\node at (1*\u,-0.8*\u) {\sz $D^2g$};
	\node at (1*\u,0.7*\u) {\sz $D^2f$};
    \coordinate (start) at (-0.6*\u,-4*\u);
    \coordinate (p1) at ({(-1/sqrt(3))*\u},{(-6/sqrt(3))*\u});
    \coordinate (p2) at ({(-0.666/sqrt(3))*\u},{(6*-0.666/sqrt(3))*\u});
    \coordinate (p3) at (0,-0.2*\u);
    \coordinate (end) at (2*\u,-0.07*\u);
    \coordinate (c1) at ($(start) + 0.05*(1*\u,7*\u)$);
    \coordinate (c2) at ($(p1) - 0.03*(1*\u,6*\u)$);
    \coordinate (c3) at ($(p1) + 0.15*(1*\u,2*\u)$);
    \coordinate (c4) at ($(p2) - 0.05*(1*\u,8*\u)$);
    \coordinate (c5) at ($(p2) + 0.05*(1*\u,8*\u)$);
    \coordinate (c6) at ($(p3) - 0.3*(1*\u,0.2*\u)$);
    \coordinate (c7) at ($(p3) + 0.2*(1*\u,0.2*\u)$);
    \coordinate (c8) at ($(end) - 0.5*(1*\u,0*\u)$);
    \draw[dashed,smooth] (start)
      .. controls (c1) and (c2) .. (p1)
      .. controls (c3) and (c4) .. (p2)
      .. controls (c5) and (c6) .. (p3)
      .. controls (c7) and (c8) .. (end);
\end{scope}
\node[below] at (current bounding box.south) {\sz Figure 3};
\end{tikzpicture}
\]
\end{exmp}

When a function is $C^n$ but on some intervals is $C^m$ for some $m>n$, Theorem \ref{t:C.n.Hoischen.C.n.fcts} can sometimes be adapted to allow better approximation of the function on the intervals where it is $C^m$. The following example illustrates this.

\begin{exmp}
Consider the closed discrete set $E=\{0\}$. Take $f(x)=x^3$ when $x\leq 0$, and let the restriction of $f$ to $[0,\infty)$ be any $C^2$ function so that $D^2f(x)$ is increasing but not constant on $[0,\infty)$ with $D^jf(0)=0$, $j=0,1,2$. Then $f$ is a $C^2$ function and there is an entire function $g$ taking real values on $\R$ so that
\begin{enumerate}
\item
$D^jg(0)=D^jf(0)=0$, $j=0,1,2$, $D^3g(0)=6$ (which is the left derivative $D^{3}_-f(0)$).

\item
$|D^jg(x)-D^jf(x)|<\ep(x)$, $x\leq 0$, $j=0,1,2,3$.

\item
$|D^jg(x)-D^jf(x)|<\ep(x)$, $x\geq 0$, $j=0,1,2$.

\item
$D^3g(x)>0$ for all $x\in\R$.

\item
$D^jg$ is comonotone with $D^jf$ with $\{0\}$ as witnessing set, $j=0,1,2$.
\end{enumerate}

\begin{proof}
We can add to the given objects in the statement of Theorem \ref{t:C.n.Hoischen.C.n.fcts} a function $\xi\colon E\to\R$ so that $\xi(a)>0$ for all $a\in E$, and require in Theorem \ref{t:C.n.Hoischen.C.n.fcts} \ref{p:C.n.Hoischen:2} that when $a\in E\sm (P^n_{\min}\cup P^n_{\max})$ is not a turning point for $D^nf$, $|D^{n+1}g(a)|=\xi(a)$. The proof requires only minor alterations to accommodate this change. (For example, in Claim \ref{claim:g_I.properties} \ref{cl.gI.2} and \ref{cl.gI.2a}, we want to say $D^{n+1}g_{I_k}(a)=s\xi(a)$.)

In our present setting, we have $n=2$ and for $x\leq 0$,  $D^nf(x)=D^2f(x)=6x$, so $P^n_{\min}=\e$. Also, for $x\geq 0$, $D^2f$ is increasing but not constant, so $0\in E\sm (P^n_{\min}\cup P^n_{\max})$ and $0$ is not a turning point for $D^2f$. Since $(Q_n)$ holds for $n\leq 3$, we can apply Theorem \ref{t:C.n.Hoischen.C.n.fcts} for $n=2$, modified as above and taking $E=\{0\}$, $\xi(0)=6$, and taking $\ep(x)/2$ in the place of $\ep(x)$, gives an entire function $h$ so that
\begin{enumerate}
\item\label{h1}
$\hspace{-5pt}{}_{h}$ $D^jh(0)=D^jf(0)$, $j=0,1,2$, and $D^3h(0)=6$.

\item\label{h2}
$\hspace{-5pt}{}_{h}$ $D^jh(x)\not=0$ when $x\not=0$, $j=1,2$.

\item\label{h3}
$\hspace{-5pt}{}_{h}$ $D^3h(x)\not=0$ for all $x\in\R$.

\item\label{h4}
$\hspace{-5pt}{}_{h}$ $|D^jh(x)-D^jf(x)|<\ep(x)/2$ for all $x\in\R$, $j=0,1,2$. (Note that $W_2=\e$.)

\item\label{h5}
$\hspace{-5pt}{}_{h}$ $D^jh$ is comonotone with $D^jf$, with $\{0\}$ as witnessing set, $j=0,1,2$.
\end{enumerate}
Let $w(x)=x^3$ for $x\leq 0$, $w(x) = h(x)$ for $x\geq 0$. Then $w$ is a $C^3$ function which retains the properties (1)$_h$\,--\,(5)$_h$ of $h$.
Apply Theorem \ref{t:piecewise.monotone.interpolation} to $w$ for $n=m=3$, $T=E=\{0\}$, and with $U_i=\e$ for $i=0,1,2,3$, to get an entire function $g$ satisfying
\begin{enumerate}
\item\label{p4}
$\hspace{-5pt}{}_{g}$ $|D^jg(x)-D^jw(x)|<\ep(x)/2$ for all $x\in\R$, $j=0,1,2,3$.

\item\label{p1}
$\hspace{-5pt}{}_{g}$ $D^jg(0)=D^jw(0)$, $j=0,1,2,3$.

\item\label{p5}
$\hspace{-5pt}{}_{g}$ $D^jg(x)$ has the same sign as $D^kw(x)$, $x\in\R$, $k=0,1,2,3$.
\end{enumerate}
From the properties (1)$_h$\,--\,(5)$_h$ of $w$ and \ref{p4}$_g$\,--\,\ref{p5}$_g$ of $g$, we see that $g$ is as desired.
\end{proof}
\end{exmp}

\end{document}